%% file: DDP_2026_4_19_Submitted.tex
\documentclass[11pt]{article}
\usepackage[margin=1in]{geometry}
\usepackage{amsmath}
\usepackage{amsfonts}
\usepackage{color}
\usepackage{latexsym}
\usepackage{tablefootnote}
\usepackage{epsfig}
\usepackage{multirow}
\usepackage{float}
\usepackage{array}
\usepackage{amssymb,amsthm}
\usepackage{hyperref}    
\usepackage{algorithm2e} 
\usepackage{cleveref}    

\newcommand*{\Q}{\mathcal{Q}}
\newcommand*{\V}{\mathcal{V}}

\newcommand{\red}[1]{{\color{red}{#1}}}

\definecolor{antiquefuchsia}{rgb}{0.57, 0.36, 0.51}
\definecolor{MyViolet}{rgb}{0.45,0.08,0.95}
\definecolor{MyBrown}{rgb}{0.45,0.08,0}
\definecolor{MyDarkBlue}{rgb}{0,0.08,0.45}

\usepackage{makecell,booktabs}
\usepackage[version=4]{mhchem}   
\usepackage[strict]{changepage}

\include{def}

\newcommand{\m}{m}

 \title{Dimension-Free Complexity Guarantees for Dual Dynamic Programming}

 \author{
 Pablo Barros \thanks{School of Applied Mathematics FGV/EMAp, 22250-900 Rio de Janeiro, Brazil. (email: {\tt pabloacbarros@gmail.com}).} 
        \qquad
		Vincent Guigues \thanks{School of Applied Mathematics FGV/EMAp, 22250-900 Rio de Janeiro, Brazil. (email: {\tt vincent.guigues@fgv.br}).} 
        \qquad
		Jiaming Liang \thanks{
        Goergen Institute for Data Science and Artificial Intelligence and Department of Computer Science, University of Rochester, Rochester, NY 14620 (email: {\tt jiaming.liang@rochester.edu}). This work was partially supported by AFOSR grant FA9550-25-1-0182.}
        \qquad
            Renato D.C. Monteiro \thanks{School of Industrial and Systems
			Engineering, Georgia Institute of
			Technology, Atlanta, GA 30332.
			(email: {\tt renato.monteiro@isye.gatech.edu}). This work
			was partially supported by AFOSR Grants FA9550-22-1-0088 and FA9550-25-1-0131.}
	}

\date{June 10, 2026}

\begin{document}

\maketitle

	\begin{abstract} 

This paper studies the complexity of a dual dynamic programming (DDP) method for solving a class of convex optimization problems with linear coupling constraints.
Existing complexity results based on DDP depend on the dimensions of the state vectors and, in particular, grow exponentially with dimension. The goal of this paper is to establish a complexity bound that is independent of the dimension.
Our approach first studies an unconstrained strongly convex problem and develops a flexible framework for DDP, called FDDP, for solving the associated dynamic programming equations, and establishes 
an iteration-complexity bound for it that is independent of the dimension.
A dimension-independent complexity bound for the original linearly constrained problem
is then obtained by
applying FDDP to a corresponding unconstrained strongly convex problem obtained via
smoothing and penalization.
Inspired by the literature on bundle methods, FDDP updates lower approximations of cost-to-go functions through a generic procedure that includes both the classical multi-cut DDP and a new two-cut DDP variant as special cases. The two-cut variant maintains only two affine cuts per stage at each iteration, making it more memory-efficient while retaining the same theoretical guarantees.
Finally, numerical experiments illustrate the practical behavior of multi-cut and two-cut DDP, including their dependence on problem parameters and their performance relative to a direct quadratic constraint reformulation.

 \end{abstract}
 
		
	 
	\section{Introduction}

This paper considers the following
optimization problem 
\begin{equation}
\label{pb:multistage0Tc}
\begin{alignedat}{2}
E^\star
:=\;& \min_{x_1,\ldots,x_T} 
&\quad& E(x):=\sum_{t=1}^T E_t(x_t) \\
& \quad \text{s.t.}
&& A_tx_t+B_tx_{t-1}=b_t,\qquad t=1,\ldots,T, \\
&&& x_t\in X_t\subseteq \mathbb R^{n_t},\qquad t=1,\ldots,T .
\end{alignedat}
\end{equation}
Here, $x_0 \in \R^{n_0}$ is fixed, $T$ denotes the number of stages, and for every $t=1,\ldots,T$, function $E_t$ is $\alpha_t$-convex with $\alpha_t \geq 0$ and set $X_t$ is nonempty, convex, and compact. 
Let $X=X_1\times\cdots\times X_T$, we say that $T$-tuple
$\bar x=(\bar x_1,\ldots,\bar x_T) \in X$
is a $\bar \varepsilon$-solution of \eqref{pb:multistage0Tc} if
\begin{equation}\label{def:solution}
    E(\bar x)-E^\star\le \bar\varepsilon, \quad \sum_{t=1}^T \|A_t\bar x_t+B_t \bar x_{t-1}-b_t\|^2 \le \bar \varepsilon^2.
\end{equation}

To solve  problem \eqref{pb:multistage0Tc}, one typically exploits its recursive structure via dynamic programming. In particular, Dual Dynamic Programming (DDP) \cite{guiguesejor17,lan2020} is a widely used decomposition method that approximates the cost-to-go functions at each stage by iteratively constructing piecewise-linear lower bounds using dual information. At a high level, DDP alternates between a forward pass, which generates trial decisions along the stages, and a backward pass, which refines the approximations of the future cost functions through cutting planes derived from dual information. This approach allows the large-scale  problem to be solved via a sequence of smaller stagewise problems.

SDDP (Stochastic Dual Dynamic Programming) is an extension of DDP
where $A_t, B_t$, and $b_t$
are random
which was introduced in \cite{per1991} as an extension of the Nested Decomposition
method \cite{birgemulti} to solve multistage stochastic linear programs.
SDDP was analyzed and extended in many references, see for instance
\cite{shapejor,Sha2012a,shadenrbook}.
Convergence analysis of SDDP and some variants
can be found in, e.g., 
\cite{lecphilgirar12,guiguessiopt2016,guigues2016isddp, gms2020isddp}.
The complexity of SDDP was studied in 
\cite{lan2020} for some convex stochastic multistage problems
while \cite{sunsddpc} studies the complexity of
an SDDP-type variant with penalization for multistage stochastic mixed-integer nonlinear optimization.
The complexity of multi-cut
DDP (the deterministic counterpart of SDDP)
to solve problem \eqref{pb:multistage0Tc}
was shown in Theorem 1 of 
\cite{lan2020}.
More precisely, it follows from Theorem~1 of \cite{lan2020}, with $\epsilon=\bar\varepsilon/T^2$
and discount factor $\lambda=1$,
 that  DDP  can compute a $\bar\varepsilon$-solution of
\eqref{pb:multistage0Tc}
in at most
\begin{equation}\label{compconvex}
(T-1)\left(\frac{DT^2}{\bar\varepsilon}+1\right)^n+1
\end{equation}
iterations, where $n = \max_{1 \le t \le T} n_t$ and $D = \max_{1 \le t \le T} D_t$.

A natural question is whether a complexity bound for DDP which is independent of the dimensions  of the 
state vectors can be established.
In this paper,
we answer positively to this question by presenting a complexity 
result independent of $n$.
For the case where the dimension $n$ is large while the number of stages $T$ is small, our complexity improves the complexity bound \eqref{compconvex} of \cite{lan2020}.
It is worth mentioning that both this paper and \cite{lan2020} provide computable termination criteria based on upper and lower bounds on the objective function.

In this paper, we also introduce a framework for DDP, referred to as FDDP, which includes the classical multi-cut DDP method and the two-cut DDP method
proposed in this paper
as special cases. 
Our goal is to establish the complexity of solving  optimization problem \eqref{pb:multistage0Tc} via FDDP,
which is designed to solve an unconstrained problem of the form
\begin{equation}\label{pb:multistage0T}
\min_{(x_1,\ldots,x_T) \in X} \sum_{t=1}^T F_t(x_t,x_{t-1}),
\end{equation}
where $F_t(\cdot,\cdot)$ is convex, $F_t(\cdot,x_{t-1})$ is $\mu_t$-strongly convex for every $x_{t-1} \in X_{t-1}$, and $F_t(x_t,\cdot)$ is $M_{t-1}$-Lipschitz continuous for every $x_t \in X_t$.
It is shown in this paper 
that the complexity, in terms of total forward/backward passes, for finding a $\bar \varepsilon$-solution of \eqref{pb:multistage0T}  is
\begin{equation}\label{eq:FDDP-cmplx}
 {\cal O}\left( \left( \frac{M^2T}{\mu \bar \varepsilon} \right)^{2
^{T-1}-1} \log\frac{1}{ \bar\varepsilon}\right),
\end{equation}
where $\mu = \min_{1 \le t \le T} \mu_t$ and $M = \max_{1 \le t \le T-1} M_t$.

To apply \eqref{eq:FDDP-cmplx} to the constrained problem \eqref{pb:multistage0Tc}, we reformulate it as an unconstrained problem by penalizing the linear coupling constraints with parameter $\rho={\cal O}(1/\bar\varepsilon)$ and adding a small quadratic perturbation with strong convexity $\mu=\Omega(\bar\varepsilon/(TD^2))$. Noting that $M={\cal O}(\rho D)={\cal O}(D/\bar\varepsilon)$ and substituting these estimates into \eqref{eq:FDDP-cmplx} gives $M^2T/(\mu\bar\varepsilon)={\cal O}(T^2D^4/\bar\varepsilon^4)$. Consequently, 
the FDDP complexity for finding a $\bar\varepsilon$-solution of \eqref{pb:multistage0Tc} is
\begin{equation}\label{eq:cmplx}
    {\cal O}\left( \left( \frac{T^2 D^4}{\bar \varepsilon^4} \right)^{2^{T-1}-1} \log\frac{1}{\bar\varepsilon}\right).
\end{equation}
As previously mentioned, this complexity does not depend on the dimensions $n_t$ of state vectors $x_t$.

In the case when problem \eqref{pb:multistage0Tc} is stochastic with random $A_t$,  $B_t,$ and $b_t$, and the objective is replaced by the expectation of the total cost, we believe that our analysis and the tool we
develop in this paper can be used to obtain a complexity
for the SDDP method
or a variant of SDDP
that also does not depend on the dimensions
of the state vectors.
We leave this question as an open
problem for future work.\\

Our contributions are as follows.\\

\par {\textbf{FDDP and Two-cut DDP.}} 
Motivated by a unified framework for various proximal bundle methods \cite{liang2021unified}, which includes multi-cut, two-cut, and single-cut lower approximation models as special instances, we introduce FDDP as a framework by providing a minimum requirement on the update scheme for lower approximation.
To the best of our knowledge, FDDP is the first DDP framework that extends the static unified framework to the dynamic setting. 
More concretely, we present two special instances of FDDP, namely, the classical multi-cut DDP method and a newly developed two-cut DDP method, which uses only two cuts for every approximate cost-to-go function. 
The benefit of studying various DDP methods through the lens of FDDP is that all instances contained in FDDP share a unified complexity analysis.
Beyond multi-cut DDP, several variants with cut selection have been studied in \cite{guiguesejor17, lan2020, bandgui21}, together with their convergence analyses. However, the cut-selection schemes in \cite{guiguesejor17,bandgui21} do not cover the two-cut DDP method considered here.\\

\par {\textbf{Complexity of FDDP.}} We first establish the complexity bound \eqref{eq:FDDP-cmplx} of FDDP for solving 
the unconstrained problem \eqref{pb:multistage0T}. Applying this bound to the original  constrained problem \eqref{pb:multistage0Tc}, we then obtain the complexity \eqref{eq:cmplx}. 
In contrast to \eqref{compconvex} obtained by \cite{lan2020},  \eqref{eq:cmplx} does not depend on the dimensions $n_t$ of the state vectors.
We believe that the analysis and the tool developed in this paper can be used to analyze the complexity of SDDP to obtain a bound 
independent on the dimension of the state vectors, contrary to \cite{lan2020}.\\

\par {\textbf{Numerical experiments.}}  We present the results of numerical
experiments where we
illustrate that DDP can be more efficient than 
a direct solution method to solve some 
strongly convex 
problems of form
\eqref{pb:multistage0Tc}. We also compare the efficiency
of DDP with two cuts only at each iteration with multi-cut DDP and illustrate
our complexity results
on instances, showing on these instances the dependence of FDDP CPU time
on parameters $n$, $T$, and $\mu_t$.\\

The outline of the paper is as follows. 
In Section  \ref{sec:analysis}, we introduce FDDP and present our 
complexity results for FDDP. In Section \ref{companalysis}, we prove the complexity of FDDP. Finally, in Section~\ref{numsim} we illustrate our results with
numerical experiments.

    \subsection{Basic definitions and notation} \label{subsec:notation}
    Let $\R$ denote the set of real numbers.
    Let $ \R_+ $ and $ \R_{++} $ denote the set of non-negative real numbers and the set of positive real numbers, respectively.
	Let $\R^n$ denote the standard $n$-dimensional Euclidean space equipped with  inner product and norm denoted by $\langle \cdot,\cdot\rangle $
	and $\|\cdot\|$, respectively. 
	Let $\log(\cdot)$ denote the natural logarithm.
	
	Let $\Psi: \R^n\rightarrow (-\infty,+\infty]$ be given. Let $\dom \Psi:=\{x \in \R^n: \Psi (x) <\infty\}$ denote the effective domain of $\Psi$.
 We say that
	$\Psi$ is proper if $\dom \Psi \ne \emptyset$.
	A proper function $\Psi: \R^n\rightarrow (-\infty,+\infty]$ is $\mu$-strongly convex for some $\mu \ge 0$ if
	$$
	\Psi\left(\alpha z+(1-\alpha) z'\right)\leq \alpha \Psi(z)+(1-\alpha)\Psi\left(z'\right) - \frac{\alpha(1-\alpha) \mu}{2}\left\|z-z'\right\|^2
	$$
	for every $z, z' \in \dom \Psi$ and $\alpha \in [0,1]$.
 The set of all proper lower semicontinuous $\mu$-strongly convex functions is denoted by $\mConv{n}$.
 When $\mu=0$, we simply denote
    $\mConv{n}$ by $\bConv{n}$.
	For $\varepsilon \ge 0$, the \emph{$\varepsilon$-subdifferential} of $ \Psi $ at $z \in \dom \Psi$ is denoted by

	\begin{equation*}
        \partial_\varepsilon \Psi (z):=\left\{ s \in\R^n: \Psi\left(z'\right)\geq \Psi(z)+\left\langle s,z'-z\right\rangle -\varepsilon, \ \forall \ z'\in\R^n\right\}.
        \end{equation*}
	The subdifferential of $\Psi$ at $z \in \dom \Psi$, denoted by $\partial \Psi (z)$, is by definition the set  $\partial_0 \Psi(z)$. We will denote
 by $\Psi'(x)$ an arbitrary subgradient
 of $\Psi$ at $x$.

Given a closed convex set $X$
and a function $g: X \rightarrow \mathbb{R}$, we denote by
${\cal B}_{X}(g)$ the set of functions
defined on $X$ which are convex, lower semicontinuous, and
which minorize $g$ everywhere on
$X$.

Finally, we denote the ball of center $x_0$ and radius 
$R$ by $B(x_0;R)$.

\section{A DDP framework for a class of strongly convex problems}\label{sec:analysis}

In this section, we propose and study our FDDP framework to solve strongly convex  problems of form 
\eqref{pb:multistage0T}.
Subsection \ref{subsec:assumption} introduces the assumptions on \eqref{pb:multistage0T} and the dynamic programming formulation of \eqref{pb:multistage0T}.
Subsection \ref{sec:fddpdesc}
describes the framework and
states the complexity of FDDP applied to \eqref{pb:multistage0T}.
Complexity of FDDP applied to
problem \eqref{pb:multistage0Tc}
is then given
in Subsection~\ref{sec:constraints}.
Finally, special instances of FDDP
are presented in Subsection \ref{sec:instances}.

\subsection{Problem formulation}\label{subsec:assumption}

For problem \eqref{pb:multistage0T}, we assume the following conditions hold
for every $t=1,\ldots,T$:
\begin{itemize}
\item[\, A0.] $X_0=\{x_0\} \subset \R^{n_0}$ and $X_t \subset \mathbb{R}^{n_t}$ is nonempty, convex, and compact;
    \item[\, A1.] $F_t: \mathbb{R}^{n_t} \times \mathbb{R}^{n_{t-1}} \to \R$ is convex and finite-valued everywhere;
\item[\, A2.] 
there exists $\mu_t>0$ such that $F_t(\cdot,x_{t-1})$ is $\mu_t$-strongly convex for every $x_{t-1} \in X_{t-1}$;
    \item[\, A3.] there exists a scalar $M_{t-1} \ge 0$ such that
    for every $(x_t,x_{t-1}) \in X_t \times X_{t-1}$, we have
\begin{equation}\label{bdsubor}
\partial_{x_{t-1}} F_t(x_t,x_{t-1}) \subset B(0;M_{t-1}).
    \end{equation}
\end{itemize}

The extension of FDDP to solve problem of form \eqref{pb:multistage0Tc} with 
coupling constraints is discussed in Subsection \ref{sec:constraints}.

Assumption A3 implies that, for every $x_t\in X_t$, the function $F_t(x_t,\cdot)$ is $M_{t-1}$-Lipschitz continuous on $X_{t-1}$, i.e.,
\begin{equation} \label{lipschF}
\left| F_t(x_t,x_{t-1}) - F_t(x_t,x_{t-1}') \right| \le M_{t-1}\left\|x_{t-1}-x_{t-1}'\right\|, \quad \forall \ x_{t-1}, x_{t-1}'\in X_{t-1}.
\end{equation}

\noindent For convenience, for $t=1,\ldots,T$, we introduce the functions
\begin{equation}\label{def:Gt}
G_{t-1}(u_t,\ldots,u_T;u_{t-1})
=
\sum_{s=t}^T F_s(u_s,u_{s-1}).
\end{equation}
It clearly follows from 
\eqref{lipschF} and \eqref{def:Gt} that
for every
$u_{\ge t} = (u_t,\ldots,u_T)\in X_t\times\cdots\times X_T$ and
every $x_{t-1}, x'_{t-1} \in X_{t-1}$,
\begin{equation}\label{lipschG}
\left|G_{t-1}(u_{\geq t};x_{t-1})-G_{t-1}(u_{\geq t};x'_{t-1})\right| \stackrel{\eqref{def:Gt}}= \left| F_t(u_t,x_{t-1}) - F_t(u_t,x_{t-1}') \right| \stackrel{\eqref{lipschF}}\leq M_{t-1}\left\|x_{t-1}-x'_{t-1}\right\|.
\end{equation}

Next, let $\Q_T:X_T\to\mathbb{R}$ denote the zero function and, for each $t=1,\ldots,T$, define the
cost-to-go function
$\Q_{t-1}:X_{t-1}\to\mathbb{R}$ by
\begin{equation}\label{secondstodpT}
\Q_{t-1}(x_{t-1})
=
\min_{(u_t,\ldots,u_T)\in X_t\times\cdots\times X_T}
G_{t-1}(u_t,\ldots,u_T;x_{t-1}),
\qquad \forall \ x_{t-1}\in X_{t-1}.
\end{equation}
It follows from definition \eqref{secondstodpT} that for every  $u_{\ge t} = (u_t,\ldots,u_T)\in X_t\times\cdots\times X_T$, we have
\begin{equation}\label{QtleGt}
\Q_{t-1}(x_{t-1})
\le
G_{t-1}\left(u_{\ge t};x_{t-1}\right),
\qquad \forall \ x_{t-1}\in X_{t-1},\quad t=1,\ldots,T.
\end{equation}
Clearly, $\Q_0(x_0)$ is the optimal value of \eqref{pb:multistage0T}. Moreover, by separating the first decision from the remaining tail, we obtain the dynamic programming recursion
\begin{equation}\label{secondstodpaT}
\Q_{t-1}(x_{t-1})
=
\min_{u_t\in X_t}
F_t(u_t,x_{t-1})+\Q_t(u_t),
\qquad \forall \ x_{t-1}\in X_{t-1},\quad t=1,\ldots,T.
\end{equation}
Therefore, an optimal solution of \eqref{pb:multistage0T} can be computed recursively as follows: set $x_0^\star=x_0$ and, for $t=1,\ldots,T$, let $x_t^\star$ be an optimal solution of \eqref{secondstodpaT} with $x_{t-1}=x_{t-1}^\star$. Then $(x_1^\star,\ldots,x_T^\star)$ is an optimal solution of \eqref{pb:multistage0T}. 

Since the functions $\Q_t$ are not available explicitly in practice, the above dynamic programming recursion cannot be implemented directly.
Instead,
the DDP approach considers
approximating 
$\Q_t(\cdot)$ in \eqref{secondstodpT}
by a relatively simple 
lower approximation
$\Gamma_t(\cdot) \in {\cal B}_{X_t}(\Q_t)$ (see Subsection \ref{subsec:notation}).
In the rest of this section, we discuss some properties of $\Q_t$
and the related subproblem obtained by replacing $\Q_t(\cdot)$ by a simple lower approximation $\Gamma_t \in {\cal B}_{X_t}(\Q_t) $, i.e., the one given by
\begin{equation}\label{firststodpa2Tb}
\V_{t-1}\left(x_{t-1};\Gamma_t\right) :=
\min_{u_t
 \in X_t} F_t(u_t,x_{t-1}) +\Gamma_t( u_t )
\end{equation}
for every $x_{t-1} \in X_{t-1}$.

\vspace*{0.5cm}

The following proposition gives some properties of $\V_t(\cdot;\Gamma_{t+1}) $.
In particular, it shows
that this function is
$M_t$-Lipschitz continuous
on $X_t$.

\begin{proposition}\label{lipV}
    For any convex function $\Gamma_{t+1}: \R^{n_{t+1}} \to \R$, the function  $\V_t(\cdot;\Gamma_{t+1})$ satisfies the following statements:
    \begin{itemize}
        \item[
        a)]
        $\V_t(\cdot;\Gamma_{t+1})$ is convex, finite everywhere, and 
        \begin{equation}\label{eq:Vt}
           \emptyset \ne \partial \V_t(x_t;\Gamma_{t+1})\subset B(0;M_t) \quad \forall \ x_t \in X_t; 
        \end{equation}
        moreover,
        \begin{equation}\label{ineq:VtMt}
            \left|\V_t(y_t;\Gamma_{t+1})-\V_t(x_t;\Gamma_{t+1})\right| \leq M_t \|y_t-x_t\| \quad \forall \ x_t, y_t \in X_t;
        \end{equation}
       \item[b)]
       for every $(y_t, x_t) \in X_t \times X_t$ and $s_t \in \partial \V_t(\cdot;\Gamma_{t+1})(x_t)$, we have
       \begin{equation}\label{ineq:Vt-linear}
           0 \le \V_t(y_t;\Gamma_{t+1}) - 
           \left[\V_t(x_t;\Gamma_{t+1})+
\langle s_t,y_t -x_t\rangle \right]
            \le 2 M_t\| y_t- x_t\|.
       \end{equation}
    \end{itemize}
\end{proposition}

\begin{proof}
    a) The first claim \eqref{eq:Vt} immediately follows from the definition of $\V_t(\cdot;\Gamma_{t+1})$ in \eqref{firststodpa2Tb}, condition A3, and Lemma~\ref{subgradValueFunc2} in the Appendix with $F=F_{t+1}$ and $\Gamma = \Gamma_{t+1}$. 
    It thus follows from the convexity of $\V_t(\cdot;\Gamma_{t+1})$ that for every $s_t \in \partial \V_t(y_t;\Gamma_{t+1})$,
    \[
    \V_t(y_t;\Gamma_{t+1})-\V_t(x_t;\Gamma_{t+1}) \le \inner{s_t}{y_t - x_t} \leq M_t \|y_t-x_t\|,
    \]
    where the second inequality is due to \eqref{eq:Vt} and the Cauchy-Schwarz inequality. Hence, the second claim \eqref{ineq:VtMt} follows by switching $x_t$ and $y_t$ in the above inequality. 
    
    b) The first inequality in \eqref{ineq:Vt-linear} follows from  the convexity of $\V_t(\cdot;\Gamma_{t+1})$. Using Cauchy-Schwarz inequality, we then have
    \[
\V_t(y_t;\Gamma_{t+1}) - 
           \left[\V_t(x_t;\Gamma_{t+1})+
\langle s_t,y_t -x_t\rangle \right] \le |\V_t(y_t;\Gamma_{t+1}) - \V_t(x_t;\Gamma_{t+1})| + \|s_t\| \|y_t - x_t\|,
    \]
    which together with \eqref{eq:Vt} and \eqref{ineq:VtMt} implies the second inequality in \eqref{ineq:Vt-linear}.
\end{proof}

\if{
\vspace*{0.5cm}
The following proposition provides elementary properties of functions
$\Q_t$:

\begin{proposition}\label{lem:Qtb}
For every $t=1,\ldots,T$, $\Q_t$ is convex, finite everywhere, and for all $x,y \in X_t$, we have
\begin{equation}\label{lipQ2T}
|\Q_t(x)-\Q_t(y)| \leq 
{M}_t \|x-y\|.
\end{equation}
\end{proposition}

\begin{proof}
    Using the definition of $\Q_t(\cdot)$ in \eqref{secondstodpT}, condition A3, and Lemma~\ref{subgradValueFunc2} in the Appendix with $F=F_{t+1}$ and $\Gamma = \Q_{t+1}$, we have that $\Q_t$ is convex and finite everywhere for every $t=1,\ldots,T$. 
    Inequality \eqref{lipQ2T} follows from an argument similar to the proof of \eqref{ineq:VtMt} in Proposition \ref{lipV}.
\end{proof}

}\fi

\subsection{Framework for DDP with $T$ stages}\label{sec:fddpdesc}

The goal of this subsection is
to describe FDDP 
and to study its complexity. 
FDDP is a framework
to solve problem \eqref{pb:multistage0T}
using Dynamic Programming equations
\eqref{secondstodpaT}.
The complexity of FDDP
is given in Theorem
\ref{prop:null-strong}.

We start describing FDDP. The statement of FDDP uses the following definition.
\begin{definition}\label{defmirror}
    Function $\overline \Gamma_t \in {\cal B}_{X_t}(\Q_t)$ is called a shadow for $\Gamma_t$ in \eqref{firststodpa2Tb} if it satisfies 
    \begin{equation}\label{eq:mirror}
        \overline \Gamma_t(x_t)=\Gamma_t(x_t) ,
    \quad x_t = \underset{u_t \in X_t}\argmin \, F_t(u_t,x_{t-1}) +\overline \Gamma_t( u_t )
    \end{equation}
    where $x_t$ is the optimal solution of \eqref{firststodpa2Tb}.
\end{definition}

An obvious example of a shadow function
for $\Gamma_t$ is 
$\Gamma_t$ itself.  Another example of a shadow function for $\Gamma_t$ will be described in detail in Subsection \ref{sec:instances}.

At iteration $k$, FDDP computes in a forward pass (Step 2)
trial points $x_1^k,\ldots,x_T^k$
which are used in a backward pass
(Step 4) to update the current lower approximations
$\Gamma_1^k$, $\ldots$, $\Gamma_T^k$
of the cost-to-go functions $\Q_1$, $\ldots$, $\Q_{T-1}$.
Different  choices of the shadow functions in Step
4 of FDDP give different instances of the framework, including the classical (multi-cut) DDP which is obtained
by choosing
$\overline \Gamma_t^k=\Gamma_t^k$, $t=1,\ldots,T-1$ (see the third comment after the description of FDDP and Subsection \ref{sec:instances}).

FDDP is given below.\\

\if{
{\bf To do list:}
\begin{itemize}
    \item 
    Fix the shadow examples
    \item Lemma 3.2
    \item
    definition of $z^k[t]$
    \item 
    Lemma 3.6
    \item 
    reorganization of the sections
\item change $z$ to $y$

\end{itemize}
}\fi

\noindent\rule[0.5ex]{1\columnwidth}{1pt}
\par {\textbf{FDDP: Framework for Dual Dynamic Programming.}}\\
\noindent\rule[0.5ex]{1\columnwidth}{1pt}
\par {\textbf{Inputs:}} 
$\bar \varepsilon>0$
and 
$(x^0_1,\ldots,x^0_T) \in X_1 \times \ldots \times X_T$
\par {\textbf{Step 1: (initialization
and computation of the initial approximations $\Gamma^1_1(\cdot),\ldots,\Gamma^1_T(\cdot)$)} 
Set $k=1$, 
$y^0:=(x^0_1,\ldots,x^0_T)$, and
$\Gamma_T^1(\cdot) \equiv 0$.
For  every $t=T-1,\ldots,1$,
 compute the value and a subgradient $s_t^0$ of the function $\V_t(\cdot;\Gamma_{t+1}^{1})$
 defined in \eqref{firststodpa2Tb} at the point $x_t^0$, and set 
\begin{equation}
 \Gamma_t^{1}(\cdot)  = \V_t\left(x_t^0;\Gamma_{t+1}^{1} \right)+
\left\langle s_t^0,\cdot -x_t^0 \right\rangle. \label{ineq:requireg}
\end{equation}
\par {\textbf{Step 2: (computing new trial points)}}
Using $\Gamma_1^k(\cdot),\ldots,\Gamma_{T-1}^k(\cdot)$,
recursively compute from 
$t=1$ to $t=T$ the pair $(m_{t-1}^k,x_t^k)$ given by
\begin{equation}\label{def:vtk}
x^k_t = \underset{u_t \in X_t}\argmin \, F_t\left(u_t,x^k_{t-1}\right) +\Gamma_t^k( u_t ), \quad
m_{t-1}^k := 
 \min_{u_t
 \in X_t} F_t\left(u_t,x^k_{t-1}\right) +\Gamma_t^k( u_t ),
\end{equation}
where $x_0^k=x_0$ and $x_0$ is as in Assumption A0; set 
$x^k=(x_1^k,\ldots,x_T^k)$
and compute $y^k=(y_1^k,\ldots,y_T^k)$ as 
\[
y^k :=
\begin{cases}
y^{k-1}, & \text{if } G_0\left(y^{k-1};x_0\right)\le G_0\left(x^k;x_0\right),\\[2mm]
x^k, & \text{otherwise}.
\end{cases}
\]
\par {\textbf{Step 3:}} {\textbf{(stopping criterion)}}
Compute
\begin{equation}\label{defp0}
p_0^k = G_0\left(y^k;x_0\right) - m_0^k.
\end{equation}
If 
$p_0^k \leq \bar \varepsilon$ then
{\bf stop} and output $\bar \varepsilon$-solution $y^k=(y_1^k,\ldots,y_T^k)$; else
go to Step 4.\\

\par {\textbf{Step 4: (updating the lower approximations $\Gamma_1^k,\ldots,\Gamma_t^k$)}} 
 Set $\Gamma_T^{k+1}(\cdot) \equiv 0$, and
 for every $t=T-1,\ldots,1$,
 let
 $\overline \Gamma_t^k(\cdot)$ be a shadow of $\Gamma_t^k(\cdot)$ for \eqref{def:vtk}, and
 compute the value and a subgradient $s_t^k$ of the function $\V_t(\cdot;\Gamma_{t+1}^{k+1})$ at the point $x_t^k$, and set 
\begin{align}
 \ell_t^{k+1}(\cdot) &= \V_t\left(x_t^k;\Gamma_{t+1}^{k+1}\right)+
\left\langle s_t^k,\cdot -x_t^k\right\rangle,\label{ineq:require1}\\
    \Gamma_t^{k+1}(\cdot) &= \max \left\{ \ell_t^{k+1}(\cdot), 
  \overline \Gamma_t^k(\cdot)\right\}. \label{ineq:require}
\end{align}
Set $k \leftarrow k+1$ and go to Step 2.

\noindent\rule[0.5ex]{1\columnwidth}{1pt}

We now make some comments about FDDP. 

First, it will be shown in Lemma \ref{lemltkv} that the model
$\Gamma_t^{k+1} : X_t \to \R$ computed in \eqref{ineq:require} satisfies $\Gamma_t^{k+1}(\cdot) \le \Q_t(\cdot)$.
Second, it will be shown in  Lemma \ref{lemstopping} that the optimal value
$\Q_0(x_0)$ of problem \eqref{pb:multistage0T} satisfies
$$
m_0^k \leq \Q_0(x_0) \leq G_0\left(y^k;x_0\right)\quad \forall \ k \geq 1,
$$
and hence  the quantity
$p_0^k=G_0(y^k;x_0)-m_0^k$ computed in \eqref{defp0} can be viewed as a primal-dual gap for the primal-dual pair $(y^k,\Gamma_1^k)$.
Hence, if FDDP stops in Step 3 because the condition $p_0^k \le \bar \varepsilon$ is satisfied,
then the solution $y^k$ output by it is a $\bar \varepsilon$-solution of \eqref{pb:multistage0T}.
Third, it follows from Definition~\ref{defmirror} that the shadow function $\overline \Gamma_t^k$ satisfies
\begin{eqnarray}
&&
{\overline \Gamma}_t^k\left(x_t^k\right)=\Gamma_t^k\left(x_t^k\right), \qquad x_t^k = \underset{u_t \in X_t}\argmin \, F_t\left(u_t,x_{t-1}^k\right) +\overline \Gamma_t^k( u_t ).
\label{condgamma10T}
\end{eqnarray}
for every
$t=1,\ldots,T-1$.
Two important ways of choosing the shadow function $\bar \Gamma_t^k$ in Step 4 are discussed in Subsection~\ref{sec:instances}.
The first one yields an approximation $\Gamma_t^{k+1}$ which is the maximum of $k+1$ affine functions while the second one yields
a two-cut approximation $\Gamma_t^{k+1}$, i.e., one consisting of the maximum of two affine functions.

Finally, we discuss how the subgradient $s_t^k$ in Step 4 can be computed.
Clearly, problem \eqref{def:vtk}
with $t$ replaced by $t+1$
is equivalent to   
\begin{equation}\label{pbdefscb}
                \min_{(u_{t+1},z_t)  \in X_{t+1} \times \R^{n_t}} \left\{F_{t+1}(u_{t+1},z_t) +\Gamma_{t+1}^{k+1}( u_{t+1} ) : z_t=x_t^k \right\}
\end{equation}
and the equivalence of statements a) and c) of Lemma \ref{subgradValueFunc1} 
in the Appendix
with $G(u_{t+1},z_t)=F_{t+1}(u_{t+1},z_t) +\Gamma_{t+1}^{k+1}( u_{t+1} )$ implies that any optimal Lagrange multiplier for the constraint $z_t=x_t^k$ is a subgradient of
$\V_t(\cdot;\Gamma_{t+1}^{k+1})$ at
$x_t^k$. Hence, any  primal-dual optimal solution of
\eqref{pbdefscb}
can be used to compute a subgradient of $\V_t(\cdot;\Gamma_{t+1}^{k+1})$ at 
$x_t^k$.

\vspace*{0.2cm}

We now state Theorem \ref{prop:null-strong}
on the complexity of FDDP
to solve problem \eqref{pb:multistage0T}.

    \begin{theorem}\label{prop:null-strong}
    Let Assumptions A0, A1, A2, and A3 hold and define 
    $\overline r$ by
\begin{equation}\label{def:barr}
\overline r := 2\sum_{t=2}^T U_t - \sum_{t=1}^{T-1} L_t,
\end{equation}
where $U_t \ge 0$ is an upper bound
for $F_t$ on the compact set $X_t \times X_{t-1}$
and $L_t < 0$ is a  lower bound for  $\Gamma_t^1$ on the compact set $X_t$.
Let
\begin{equation}\label{defQeps}
\overline Q:=40 \sum_{t=1}^{T-1} \frac{M_t^2 }{\mu_t}.
\end{equation}
    Then the number of iterations performed by
    FDDP until it obtains a $\bar \varepsilon$-solution of \eqref{pb:multistage0T}, i.e., a tuple $(y_1^k,\ldots,y_T^k) \in X_1 \times \ldots \times X_T$
    satisfying $p_0^k \leq \bar \varepsilon$,
    is bounded by
    \begin{equation}\label{ourcfddp}
	     \left(
            1+ \left [ \max\left(2,\frac{\overline Q}{\bar \varepsilon}\right)  \right]^{2^{T-1}-1} \right) \log\left( \frac{2 {\overline r}}{\bar \varepsilon}\right) + 1.
	    \end{equation}
	\end{theorem}

We note that, in contrast with 
the complexity result of 
DDP from \cite{lan2020}
given by
\eqref{compconvex}, our complexity bound
\eqref{ourcfddp}
for FDDP does not depend on an upper bound $n$ on the dimensions of the state vectors. Therefore, when $T$ is moderate and $n$ is large, our bound is better than \eqref{compconvex}.
We also observe that
the complexity decreases as the strong convexity constants $\mu_t$ increase and decreases as Lipschitz constants $M_t$
decrease.

\subsection{Complexity of FDDP applied to problem \eqref{pb:multistage0Tc}} \label{sec:constraints}

We now study FDDP applied to constrained  problems of form
\eqref{pb:multistage0Tc}. The argument is to replace the coupling constraints
by a quadratic penalty, add a small strongly convex perturbation, and then apply Theorem~\ref{prop:null-strong}. We show that the resulting point is a $\bar\varepsilon$-solution of the original constrained problem in the sense of
\eqref{def:solution}, namely, it has objective error and coupling-constraint
violation at most $\bar\varepsilon$.

Set $x:=(x_1,\ldots,x_T)$ and $X:=X_1\times\cdots\times X_T$, and define
\begin{equation*}
{\cal A}x-b
:=
\begin{pmatrix}
A_1x_1+B_1x_0-b_1\\
A_2x_2+B_2x_1-b_2\\
\vdots\\
A_Tx_T+B_Tx_{T-1}-b_T
\end{pmatrix}.
\end{equation*}
Then, \eqref{pb:multistage0Tc} is equivalent to 
\begin{equation} \label{pb:constraint-compact}
E^\star
=
\min\{E(x):{\cal A}x=b,\ x\in X\}.
\end{equation}

For $\rho>0$, we consider the penalized problem
\begin{equation}\label{penalized}
E_\rho^\star
=
\min_{(x_1,\ldots,x_T) \in X}
\sum_{t=1}^T
\left[E_t(x_t) + \frac{\rho}{2} \|A_tx_t+B_tx_{t-1}-b_t\|^2\right].
\end{equation}
To recover the stagewise strong convexity required by
Theorem~\ref{prop:null-strong}, we use the perturbed penalized stage functions
\begin{equation}
\label{def:penalized-perturbed-stage}
E_t^{\bar \varepsilon}(x_t,x_{t-1})
:=
E_t(x_t)
+
\frac{\rho}{2}\|A_tx_t+B_tx_{t-1}-b_t\|^2
+
\frac{\bar\varepsilon}{2TD_t^2}
\left\|x_t-x_t^0\right\|^2,
\end{equation}
where, for $t=1,\ldots,T$, $D_t$ is the diameter of $X_t$ and
$x_t^0\in X_t$ is arbitrary. Finally, we apply FDDP to the
perturbed penalized problem 
\begin{equation}
\label{pb:perturbed-penalized}
E_{\rho,\bar \varepsilon}^\star
:=
\min_{(x_1,\ldots,x_T) \in X}
\sum_{t=1}^T
E_t^{\bar \varepsilon}(x_t,x_{t-1}).
\end{equation}
We show that every
$\bar\varepsilon/2$-solution of the perturbed penalized problem
\eqref{pb:perturbed-penalized} is a $\bar\varepsilon$-solution of the penalized
problem \eqref{penalized}. Hence, it follows from Lemma~\ref{LemApp3} in Appendix \ref{App:tech} that a $\bar \varepsilon$-solution of \eqref{penalized} is also a $\bar \varepsilon$-solution to \eqref{pb:multistage0Tc} in the sense of \eqref{def:solution}. As a result, it suffices to find a $\bar\varepsilon/2$-solution of the penalized perturbed problem \eqref{pb:perturbed-penalized} using FDDP.

\begin{theorem}
\label{thm:constraints-complexity}
Let $\bar \varepsilon>0$ be given and $p^\star$ be 
an optimal Lagrange multiplier
for constraint
$\mathcal{A}x=b$ in
\eqref{pb:constraint-compact}. 
For $\rho\ge (4\|p^\star\|+8)/\bar\varepsilon$, FDDP applied to \eqref{pb:perturbed-penalized} computes a $\bar \varepsilon$-solution of \eqref{pb:multistage0Tc} in at most
\begin{equation}
\label{eq:constraints-complexity}
    {\cal O}\left(
\left(
\frac{T^2D^4}{\bar\varepsilon^4}
\right)^{2^{T-1}-1}
\log\left(\frac{1}{\bar\varepsilon}\right)
\right)
\end{equation}
iterations, where $D = \max_{1 \le t \le T} D_t$.
\end{theorem}

\begin{proof}
Since the perturbation $\bar\varepsilon \|x_t-x_t^0\|^2/(2TD_t^2)$ in
\eqref{def:penalized-perturbed-stage} is nonnegative, we have $E_\rho^\star\le E_{\rho,\bar \varepsilon}^\star$
where 
$E_\rho^\star$ and $E_{\rho,\bar \varepsilon}^\star$ are given by
\eqref{penalized} and \eqref{pb:perturbed-penalized}, respectively.
Moreover, if $x_\rho^\star$ solves \eqref{penalized}, then
\[
E_{\rho,\bar \varepsilon}^\star
\le
\sum_{t=1}^T
E_t^{\bar\varepsilon}
\left(x_{\rho,t}^\star,x_{\rho,t-1}^\star \right)
\stackrel{\eqref{def:penalized-perturbed-stage}}=
E_\rho^\star
+
\sum_{t=1}^T
\frac{\bar\varepsilon}{2TD_t^2}
\left\|x_{\rho,t}^\star-x_t^0 \right\|^2
\le
E_\rho^\star+\frac{\bar\varepsilon}{2}.
\]
Hence, every $\bar\varepsilon/2$-solution of
\eqref{pb:perturbed-penalized} is a $\bar\varepsilon$-solution of
\eqref{penalized}. Lemma~\ref{LemApp3}, applied with
\eqref{probf} replaced by \eqref{pb:constraint-compact}, then yields a $\bar\varepsilon$-solution of the original constrained problem \eqref{pb:constraint-compact}.

Next, we apply Theorem~\ref{prop:null-strong}
to compute a $\bar\varepsilon/2$-solution of \eqref{penalized}
and bound the corresponding number of FDDP iterations. We first record
the constants entering the theorem. For $t=1,\ldots,T$, if
$U_t^F\ge 0$ is an upper bound for $E_t$ on $X_t$, define
\begin{equation}
\label{def:Ht-constraints}
\tilde U_t
:=
U_t^F+\frac{\rho}{2}H_t^2+\frac{\bar\varepsilon}{2T},
\qquad
H_t
:=
\max_{(x_t,x_{t-1})\in X_t\times X_{t-1}}
\|A_tx_t+B_tx_{t-1}-b_t\|.
\end{equation}
Then $\tilde U_t$ is an upper bound for $E_t^{\bar\varepsilon}$ on
$X_t\times X_{t-1}$. Let $\tilde\Gamma_t^1$ be the initial lower approximation
constructed by FDDP for \eqref{pb:perturbed-penalized}, let $\tilde L_t<0$ be a
lower bound for $\tilde\Gamma_t^1$ on $X_t$, and set
\[
\tilde r
:=
2\sum_{t=2}^T \tilde U_t-\sum_{t=1}^{T-1}\tilde L_t .
\]

We now verify the assumptions of Theorem~\ref{prop:null-strong}. Assumption A0
is unchanged. Assumption A1 follows from the convexity and finite-valuedness of
$E_t$ and from the convexity of the added quadratic terms. The strong convexity
modulus of $E_t^{\bar\varepsilon}(\cdot,x_{t-1})$ is
\begin{equation}
\label{def:mut-constraints}
\mu_t
=
\alpha_t
+
\rho\lambda_{\min}(A_t^\top A_t)
+
\frac{\bar\varepsilon}{TD_t^2}
>0,
\end{equation}
so Assumption A2 holds. Finally, by the definitions \eqref{def:penalized-perturbed-stage} and
\eqref{def:Ht-constraints}, we have for $t=1,\ldots,T-1$,
\begin{equation}
\label{def:Mt-constraints}
\partial_{x_t}E_{t+1}^{\bar\varepsilon}(x_{t+1},x_t)
\subseteq B(0; M_t),
\qquad
M_t
:=
\rho\|B_{t+1}\|H_{t+1},
\end{equation}
thus Assumption A3 holds as well.

We are now in a position to apply Theorem~\ref{prop:null-strong} to
\eqref{pb:perturbed-penalized} with accuracy $\bar\varepsilon/2$. This gives iteration complexity
\begin{equation}
\label{eq:exact-complexity-constraints}
\left(
1+
\left[
\max\left\{
2,\frac{2\overline Q}{\bar\varepsilon}
\right\}
\right]^{2^{T-1}-1}
\right)
\log\left(\frac{4\tilde r}{\bar\varepsilon}\right)+1,
\end{equation}
where
\begin{equation}
\label{eq:Qbar-order-constraints}
\frac{2\overline Q}{\bar\varepsilon}
=
\frac{80}{\bar\varepsilon}\sum_{t=1}^{T-1}\frac{M_t^2}{\mu_t}
=
\frac{80}{\bar\varepsilon} \sum_{t=1}^{T-1}
\frac{
\rho^2\|B_{t+1}\|^2H_{t+1}^2
}{\alpha_t
+
\rho\lambda_{\min}(A_t^\top A_t)
+
\frac{\bar\varepsilon}{TD_t^2}
}
=
{\cal O}\left(\frac{T^2D^4}{\bar\varepsilon^4}\right),
\end{equation}
and we used $H_t = {\cal O}(D).$ Finally, \eqref{eq:constraints-complexity} follows from combining \eqref{eq:exact-complexity-constraints} and
\eqref{eq:Qbar-order-constraints}.
\end{proof}

\subsection{Special instances of FDDP}\label{sec:instances}

We now consider two specific implementations of the FDDP framework:
a multi-cut DDP 
that uses $k$ cuts for the approximations at iteration $k$ and 
a two-cut DDP
that uses two cuts for the approximations $\Gamma_t^k$ at every iteration. It suffices to describe $\overline \Gamma_t^k$ chosen in \eqref{ineq:require} to update $\Gamma_t^{k+1}$.

\vgap

\textbf{Multi-cut DDP.}  This special instance of FDDP, referred to as multi-cut DDP, chooses
$\overline \Gamma_t^k = \Gamma_t^k$ 
for every $t = T-1, \ldots, 1$.
Clearly, it follows from Definition \ref{defmirror}, or the remark following it, that $\overline \Gamma_t^k$ is a shadow of $\Gamma_t^k$.
In view of \eqref{ineq:requireg} and \eqref{ineq:require}, it can be easily seen that 
$\Gamma_t^k$ is the maximum of $k$ affine functions. 


\vgap

We will now describe a two-cut instance of FDDP, which we refer to as the two-cut DDP.

\par {\textbf{Two-cut DDP.}}
Before describing how the two-cut DDP generates its sequence of shadows, we first observe that for any $k \ge 2$,
\eqref{ineq:require} with $k$ replaced by $k-1$ implies that problem \eqref{def:vtk} is equivalent to
\begin{equation}\label{eq:vtk}
    m_{t-1}^k
=\min_{u_t \in X_t, \gamma_t \in \R} \left\{F_t\left(u_t,x_{t-1}^k\right) +\gamma_t:
\overline \Gamma_t^{k-1}(u_t) - \gamma_t \le 0, \, \ell_t^k(u_t) - \gamma_t \le 0 \right\}.
\end{equation}
For $t = T-1, \ldots, 1$, the the two-cut DDP instance generates its shadow sequence $\{\bar \Gamma_t^k\}$ as  
\begin{equation}\label{def:Atk}
\overline \Gamma_t^k(\cdot) = \left\{
\begin{array}{ll}
\Gamma_t^1(\cdot),& \mbox{ if }k=1,\\
   \beta_t^k \, \overline \Gamma_t^{k-1}(\cdot) + \left(1-\beta_t^k\right) \, \ell_t^k(\cdot),& \mbox{ if }k \geq 2,
    \end{array}
    \right.
    \end{equation} 
where $\Gamma_t^1$ as in \eqref{ineq:requireg} and $\beta_t^k\in [0,1]$ for $k \ge 2$
is the Lagrangian multiplier for the first constraint 
$\overline \Gamma_t^{k-1}(u_t) - \gamma_t \le 0$ of \eqref{eq:vtk}.
It follows from \eqref{ineq:requireg} and \eqref{def:Atk} that
$\overline \Gamma_t^k$ is an affine function for every $k \ge 1$, and hence that
$\Gamma_t^k$ for $k \ge 2$ is the maximum of two affine functions, because of \eqref{ineq:require}.

The following proposition, whose proof is given in the appendix, shows that the two-cut DDP instance
is a special implementation of FDDP.




\begin{proposition}\label{twocutsp}
    The affine function $\overline \Gamma_t^k$ defined in \eqref{def:Atk} is a shadow of $\Gamma_t^k$ for \eqref{def:vtk}.
\end{proposition}

\section{Complexity analysis of FDDP}\label{companalysis}

The main objective of this section is to prove Theorem \ref{prop:null-strong}, which establishes the iteration complexity of FDDP for obtaining a $\bar\varepsilon$-optimal solution of \eqref{pb:multistage0T}. We first derive a number of auxiliary results for the sequences generated by the algorithm.

In what follows, for every $t \in \{0,\ldots,T\}$ and $k \geq 1$, let
\begin{equation}\label{Deltak}
\Delta_t^k := \left\|x_t^{k+1}-x_t^k\right\|.
\end{equation}
Since $x_0^k=x_0$ for all $k \geq 1$ (see Step~2 of FDDP), we have
\begin{equation}\label{delta0eq0}
\Delta_0^k=0
\end{equation}
for all $k \geq 1$.

Lemma~\ref{lemltkv} summarizes the main properties of $\V_t(\cdot;\Gamma_{t+1}^{k+1})$, $\ell_t^{k+1}$, and $\Gamma_t^k$ used in what follows. In particular,  it shows that $\ell_t^{k+1}$ and $\Gamma_t^k$ are convex lower approximations of $\Q_t$ on $X_t$.

\begin{lemma}\label{lemltkv}For every $k\geq 1$, the following statements hold:
\begin{itemize}
\item[(i)] $\V_t(\cdot;\Gamma_{t+1}^{k+1})$ is convex and  $M_t$-Lipschitz continuous;
\item[(ii)] 
for every $u_t \in X_t$, 
we have
\beq \label{eq:sim-cond1}
\V_t\left(u_t;\Gamma^{k+1}_{t+1}\right) -  \ell_t^{k+1}(u_t) \le 2 M_t \left\|u_t-x_t^k\right\|,
\eeq
and
\begin{equation}\label{eq:sim-cond2b}
\Q_t(u_t) \geq
\V_t\left(u_t;\Gamma_{t+1}^{k+1}\right) \geq 
\ell_t^{k+1}(u_t),
\end{equation}
where $\V_t(\cdot;\cdot)$ and $\ell_t^{k+1}(\cdot)$ are defined in
\eqref{firststodpa2Tb} and \eqref{ineq:require1}, respectively.
As a consequence, we have
\beq \label{eq:sim-cond1'} m_t^{k+1} -
\ell_t^{k+1}\left(x_t^{k+1}\right) 
\le 2 M_t \Delta_t^k,
\eeq
where $\Delta_t^k$ is given by \eqref{Deltak}
and 
\begin{equation}\label{eq:sim-cond2b-xt}
\Q_t\left(x_t^{k+1}\right) \geq
m_t^{k+1} \geq 
\ell_t^{k+1}\left(x_t^{k+1}\right);
\end{equation}
\item[(iii)] for all $t=1,\ldots,T$, and all $k \geq 1$, we have that 
$\Gamma_t^k$ is convex
and satisfies
\begin{equation}\label{gammaleQ}
\Gamma_t^k \leq \mathcal{Q}_t.
\end{equation}
\end{itemize}
 \end{lemma}

\begin{proof}
(i)  This statement follows directly from Proposition~\ref{lipV}(a) applied to $\Gamma_{t+1}^{k+1}$.

ii) By definition, $s_t^k\in \partial \V_t(x_t^k;\Gamma_{t+1}^{k+1})$ and
$\ell_t^{k+1}(\cdot)=\V_t(x_t^k;\Gamma_{t+1}^{k+1})+\langle s_t^k,\cdot-x_t^k\rangle$.
Proposition~\ref{lipV}(b) gives for every $u_t\in X_t$,
\[
0\le \V_t\left(u_t;\Gamma_{t+1}^{k+1}\right)-\ell_t^{k+1}(u_t)\le 2M_t\left\|u_t-x_t^k\right\|,
\]
which implies \eqref{eq:sim-cond1} and $\V_t(u_t;\Gamma_{t+1}^{k+1})\ge \ell_t^{k+1}(u_t)$.
Moreover, since $\Gamma_{t+1}^{k+1}\le \Q_{t+1}$ on $X_{t+1}$, we have for all $u_t\in X_t$,
\begin{align*}
    \V_t\left(u_t;\Gamma_{t+1}^{k+1}\right)
&=\min_{u_{t+1}\in X_{t+1}}\left(F_{t+1}(u_{t+1},u_t)+\Gamma_{t+1}^{k+1}(u_{t+1})\right)
\\&\le
\min_{u_{t+1}\in X_{t+1}}\left(F_{t+1}(u_{t+1},u_t)+\Q_{t+1}(u_{t+1})\right)
\\&=\Q_t(u_t),
\end{align*} 
using \eqref{secondstodpT} in the last equality. This proves \eqref{eq:sim-cond2b}.
Taking $u_t=x_t^{k+1}$ in \eqref{eq:sim-cond1} and using $m_t^{k+1}=\V_t(x_t^{k+1};\Gamma_{t+1}^{k+1})$
gives \eqref{eq:sim-cond1'}. Taking $u_t=x_t^{k+1}$ in
\eqref{eq:sim-cond2b} gives \eqref{eq:sim-cond2b-xt}.

(iii) By (i), for
$t=1,\ldots,T-1$, 
$\V_t(\cdot;\Gamma_{t+1}^{k+1})$
is convex and
given $s_t^k \in \partial \V_t(x_t^k;\Gamma_{t+1}^{k+1})$, we can compute its linearization 
$$
 \ell_t^{k+1}(\cdot) = \ell_t\left(\cdot;x_t^k,s_t^k,\Gamma_{t+1}^{k+1}\right):=\V_t\left(x_t^k;\Gamma_{t+1}^{k+1}\right)+
\left\langle s_t^k,\cdot -x_t^k \right\rangle 
$$ which is below the functions
$\V_t(\cdot;\Gamma_{t+1}^{k+1})$
and $\Q_t$ on $X_t$. 
Moreover, by definition of a shadow, $\overline \Gamma_t^k$
is also below $\Q_t$ on 
$X_t$ which implies,
using definition \eqref{ineq:require}
of $\Gamma_t^{k+1}$, that
$\Gamma_t^{k+1}$ is convex and
satisfies \eqref{gammaleQ}.
\end{proof}

The next result confirms that the stopping criterion in Step~3 of FDDP is well defined.

\begin{lemma}\label{lemstopping} For every $k \geq 1$, the following inequalities hold: 
\begin{align}
&\Q_0(x_0) \le G_0\left(y^k;x_0\right), \nn \\
& m_{t-1}^k \le \Q_{t-1}\left(x_{t-1}^k\right), \quad \forall \ t \ge 1. \label{mtleQt}
\end{align}
As a consequence, 
the optimal value $\Q_0(x_0)$ of \eqref{pb:multistage0T} lies in the interval $[m_0^k, G_0(y^k;x_0)] $. Moreover, if 
$p_0^k$ computed in step 3 of FDDP satisfies
$p_0^k \le \bar \varepsilon$, then $y^k$ is a $\bar \varepsilon$-solution of \eqref{pb:multistage0T}.
\end{lemma}
\begin{proof} 
For every $t\in\{1,\ldots,T\}$,
\begin{equation*}
m_{t-1}^k
\stackrel{\eqref{def:vtk}}{=}
\min_{u_t\in X_t}\Big\{F_t\left(u_t,x_{t-1}^k\right)+\Gamma_t^k(u_t)\Big\}
\stackrel{\eqref{gammaleQ}}{\le}
\min_{u_t\in X_t}\Big\{F_t\left(u_t,x_{t-1}^k\right)+\Q_t(u_t)\Big\}
\stackrel{\eqref{secondstodpT}}{=}
\Q_{t-1}\left(x_{t-1}^k\right).
\end{equation*}
Moreover, we have $
\Q_0(x_0)
\le
G_0(y^k;x_0)$ by \eqref{QtleGt}.
\end{proof}

For notational simplicity, for fixed $k \geq 1$ and $t \in \{1,\ldots,T-1\}$, we write
\begin{align}
\left(x_t,x_t^+\right) &:= \left(x_t^k,x_t^{k+1}\right), 
\qquad
\left(m_t,m_t^+\right) := \left(m_t^k,m_t^{k+1}\right), 
\qquad
\overline{\Gamma}_t(\cdot) := \overline{\Gamma}_t^k(\cdot), \label{shortnot1}\\
\Gamma_t^+(\cdot) &:= \Gamma_t^{k+1}(\cdot),
\qquad
\ell_t^+(\cdot) := \ell_t^{k+1}(\cdot),
\qquad
\Delta_t := \Delta_t^k. \label{shortnot2}
\end{align}

The next result gives a recursive inequality for $\{m_t^k\}$.
\begin{lemma}\label{lemfirstq0m}
For every $k\geq 1$,
$t \in \{1,\ldots,T-1\}$, and $\tau \in [0,1]$, we have
\begin{equation}\label{ineqlemma}
    m_{t-1}^{k+1}
-\tau m_{t-1}^k
\ge (1-\tau) \left [
 F_t\left(x_t^{k+1},x_{t-1}^{k+1}\right) + m_t^{k+1}
 -2{M}_t \Delta_t^k \right] 
+\tau\left[ \frac{\mu_t}{2}\left(\Delta_t^k\right)^2 - 
M_{t-1}\Delta_{t-1}^k \right],
\end{equation}
where $m_{t-1}^k$
is as in \eqref{def:vtk}, and
$\Delta_t^k$
is as in \eqref{Deltak}, with $\Delta_0^k=0$ by \eqref{delta0eq0}.
\end{lemma}

\begin{proof} Using notation \eqref{shortnot1}-\eqref{shortnot2}, inequality \eqref{ineqlemma} then becomes
\begin{equation}\label{ineq:goal}
    m_{t-1}^+ -\tau m_{t-1} \ge (1-\tau) \left [
 F_t\left(x_t^+,x_{t-1}^+\right) + m_t^+
 -2{M}_t \Delta_t \right] 
+\tau\left[ \frac{\mu_t}{2}(\Delta_t)^2 - 
M_{t-1}\Delta_{t-1} \right].
\end{equation}
Since $x_t$ is the optimal solution of \eqref{condgamma10T}, whose objective function is $\mu_t$-strongly convex in view of Assumption A2, we have for every $u_t \in X_t$,
\begin{equation}\label{sconvstep00b}
F_t(u_t,x_{t-1}) + \overline \Gamma_t(u_t) \geq F_t(x_t,x_{t-1}) + \overline \Gamma_t(x_t)+\frac{\mu_t}{2}\|u_t-x_t\|^2.
\end{equation}
The above inequality with $u_t=x_t^+$, the definition of $\Delta_t$ in \eqref{Deltak}, and the identities in
\eqref{def:vtk} and \eqref{condgamma10T}, imply that
\begin{align}
F_t\left(x_t^+,x_{t-1}\right) + \overline \Gamma_t\left(x_t^+\right) & \stackrel{\eqref{Deltak},\eqref{sconvstep00b}}{\geq}   F_t(x_t,x_{t-1}) + \overline \Gamma_t(x_t)+\frac{\mu_t}{2}(\Delta_t)^2 \nn \\
 & \stackrel{\ \ \, \eqref{condgamma10T}\ \ \, }{=} F_t(x_t,x_{t-1}) + \Gamma_t(x_t)+\frac{\mu_t}{2}(\Delta_t)^2
 \stackrel{\eqref{def:vtk}}  = \m_{t-1}+\frac{\mu_t}{2} (\Delta_t)^2.\label{sconvstep0b}
\end{align}
It follows from \eqref{def:vtk} and \eqref{ineq:require} that for every $\tau \in [0,1]$,
\begin{align*}
m_{t-1}^+ 
&\stackrel{\eqref{def:vtk}}=  F_t\left(x_t^+,x_{t-1}^+\right)  + \Gamma_t^+ \left(x_t^+\right)
\stackrel{\eqref{ineq:require}}\geq F_t\left(x_t^+,x_{t-1}^+\right)+ (1-\tau)
\ell_t^+\left(x_t^+\right)+
\tau {\overline \Gamma}_t\left(x_t^+\right)\\
& \ = \, (1-\tau)\left[F_t\left(x_t^+,x_{t-1}^+\right) + 
 \ell_t^+\left(x_t^+\right) \right]+
\tau \left[F_t\left(x_t^+,x_{t-1}^+\right)+
{\overline \Gamma}_t\left(x_t^+\right) \right]\\
& \ \geq \, (1-\tau)\left[F_t\left(x_t^+,x_{t-1}^+\right)+\ell_t^+\left(x_t^+\right)\right]+
\tau \left[
F_t\left(x_t^+,x_{t-1}\right) + {\overline \Gamma}_t\left(x_t^+\right) -{M}_{t-1} \Delta_{t-1}\right],
\end{align*}
where the last inequality is due to the fact that $F_t(x_t^+,\cdot)$ is $M_{t-1}$-Lipschitz continuous.
Plugging \eqref{sconvstep0b} into the above inequality, we obtain
\[
m_{t-1}^+ \stackrel{\eqref{sconvstep0b}}{\geq} (1-\tau)\left[F_t\left(x_t^+,x_{t-1}^+\right)+\ell_t^+\left(x_t^+\right)\right]+
\tau \left[
\m_{t-1} + \frac{\mu_t}{2}(\Delta_t)^2 -M_{t-1} \Delta_{t-1}\right].
\]
Finally, this inequality and \eqref{eq:sim-cond1'} imply that \eqref{ineq:goal} holds.
\end{proof}

Recall from step 3 of FDDP that the algorithm
stops when 
$p_0^k \leq \bar \varepsilon$,
where $p_0^k$ is defined
in  \eqref{defp0}.
To analyze \(p_0^k\), we introduce for each stage \(t=2,\ldots,T-1\), a quantity \(p_{t-1}^k\) that plays  the same role for stage \(t\) as \(p_0^k\) does for the first stage. More specifically, for every 
$t \in \{1,\ldots,T-1\}$, first  define
$y^k[t]=(y_t^k[t],\ldots,y_T^k[t]) \in X_t \times \ldots \times X_T$ for every $k \ge 1$ as
\begin{equation}\label{initinduc2}
y^{k}[t]=
\left\{
\begin{array}{ll}
y^{k-1}[t],& \mbox{if }G_{t-1}\left(y^{k-1}[t];x_{t-1}^{k}\right) \leq G_{t-1}\left(x_{\ge t}^{k};x_{t-1}^{k}\right),\\
x_{\ge t}^{k}, & \mbox{otherwise,}
\end{array}
\right.
\end{equation}
where $G_{t-1}$ is defined in \eqref{def:Gt}, $ x_{\ge t}^k :=(x_t^k,x_{t+1}^k,\ldots,x_T^k)$, and by convention $y^0[t]= y^0_{\ge t}$.
A straightforward observation of \eqref{initinduc2} is that for every $k \ge 1$,
\begin{equation}\label{ineq:obs-G}
    G_{t-1}\left(y^{k}[t];x^{k}_{t-1}\right)  \le \min \left\{ G_{t-1}\left(y^{k-1}[t];x^{k}_{t-1}\right), \, G_{t-1}\left(x_{\ge t}^{k};x^{k}_{t-1}\right) \right\} .
\end{equation}

Next, define for every $k \ge 1$ and $t=1,\ldots,T-1$,
\begin{equation}\label{def:ptk}
    p_{t-1}^k :=G_{t-1}\left(y^k[t];x^k_{t-1}\right)-m_{t-1}^k,
\end{equation}
and observe that $p^k_0$ is the same quantity as in
\eqref{defp0}, in view of \eqref{initinduc2} with $t=1$.
We also observe that for every $k\ge 1$ and $t = 1,\ldots,T-1$,  
\begin{equation}\label{p1ineq0}
p_{t-1}^k \ge 0.
\end{equation}
Indeed, it follows from \eqref{mtleQt} and \eqref{QtleGt} with $u_{\ge t} = y^k[t]$ that
\[
m_{t-1}^k \stackrel{\eqref{mtleQt}}\le \ \Q_{t-1}\left(x^k_{t-1}\right) \stackrel{\eqref{QtleGt}}\le \, G_{t-1}\left(y^k[t];x_{t-1}^k\right).
\]
Hence, \eqref{p1ineq0} holds in view of the definition of $p_{t-1}^k$ in \eqref{def:ptk}.






The next lemma shows that, for any $t = 1,\ldots,T-1$,
$\{p_{t-1}^k\}$ is a sequence of primal-dual gaps for problem \eqref{secondstodpaT} with $x_{t-1}=x_{t-1}^k$ and
provides recursive
relations satisfied by it.

\begin{lemma}\label{ptk}
For every $k\ge 1$ and $t\in\{1,\ldots,T-1\}$, define
\begin{equation}\label{defetk}
 e_t^{k+1} = G_t\left(x_{\ge t+1}^{k+1};x^{k+1}_t\right) - m_t^{k+1},
\end{equation}
then for every $\tau \in (0,1)$, we have
\begin{equation}\label{p1ineq}
p_{t-1}^{k+1}-\tau p_{t-1}^k \le  
\left(1-\tau\right)\left(e_t^{k+1}+2M_t\Delta_t^k\right)
-\frac{\tau\mu_t}{2}\left(\Delta_t^k\right)^2
+2 \tau M_{t-1}\Delta_{t-1}^k.
\end{equation}
\end{lemma}

\begin{proof} In what follows, for fixed $k$, 
on top of notation \eqref{shortnot1}, \eqref{shortnot2}, we use the shorthand notation
\begin{equation}\label{shorthand}
\begin{aligned}
\left(e_t^+,e_t\right):=\left(e_t^{k+1},e_t^k\right),\;\left(p_{t-1}^+,p_{t-1}\right):=\left(p_{t-1}^{k+1},p_{t-1}^k\right),\;\left(y^+[t],y[t]\right):=\left(y^{k+1}[t],y^k[t]\right).
\end{aligned}
\end{equation}

We fix $k \geq 1$ and $t \in \{1,\ldots,T\}$, and
use  the shorthand notation
\eqref{shorthand}.
It follows from the definitions of $G_{t-1}$ and $e_t^k$  in \eqref{def:Gt} and \eqref{defetk}, respectively,  that
\begin{align*}
    F_t\left(x_t^+,x_{t-1}^+\right) + m_t^+ &\stackrel{\eqref{defetk}} = F_t\left(x_t^+,x_{t-1}^+\right) + G_t\left(x_{\ge t+1}^+;x^+_t\right) -e_t^+ \stackrel{\eqref{def:Gt}}= G_{t-1}\left(x_{\ge t}^+;x^+_{t-1}\right) -e_t^+.
\end{align*}
Combining the above equation and Lemma \ref{lemfirstq0m}, we obtain
\begin{align}
    m_{t-1}^+ -\tau m_{t-1} - \frac{\tau \mu_t}{2}\left(\Delta_t\right)^2 + \tau M_{t-1}\Delta_{t-1} & \stackrel{\eqref{ineq:goal}}\ge \left(1-\tau\right) \left\{
 F_t\left(x_t^+,x_{t-1}^+\right) + m_t^+
 -2{M}_t \Delta_t \right\} \nn \\
& \ = \left(1-\tau\right) \left \{G_{t-1}\left(x_{\ge t}^+;x^+_{t-1}\right) -e_t^+
 -2{M}_t \Delta_t \right\}. \label{ineq:partial} 
\end{align}
Using the definition of $p_{t-1}^k$ in \eqref{def:ptk} and observation \eqref{ineq:obs-G}, we have
\begin{align*}
& p_{t-1}^+-\tau p_{t-1} \\
& \stackrel{\eqref{def:ptk}}= \left\{ G_{t-1}\left(y^+[t];x^+_{t-1}\right) - m_{t-1}^+ \right\}  -\tau \left \{G_{t-1}\left(y[t];x_{t-1}\right) - m_{t-1} \right\} \\
&= \left\{ G_{t-1}\left(y^+[t];x^+_{t-1}\right) - \tau G_{t-1}\left(y[t];x^+_{t-1}\right) \right\} + \tau \left\{G_{t-1}\left(y[t];x^+_{t-1}\right) - G_{t-1}\left(y[t];x_{t-1}\right) \right\}\\
& \ \ \quad - \left(m_{t-1}^+-\tau m_{t-1}\right) \\
&\stackrel{\eqref{ineq:obs-G}}\leq  \left(1-\tau\right) G_{t-1}\left(x_{\ge t}^+;x^+_{t-1}\right) + \tau \left\{G_{t-1}\left(y[t];x^+_{t-1}\right) - G_{t-1}\left(y[t];x_{t-1}\right) \right\}\\ 
& \ \ \quad - \left(m_{t-1}^+-\tau m_{t-1}\right)\\
&\stackrel{\eqref{ineq:partial}}
\le \tau \left\{G_{t-1}\left(y[t];x^+_{t-1}\right) - G_{t-1}\left(y[t];x_{t-1}\right) \right\} + \left(1-\tau\right) \left \{
  e_t^+ +
 2{M}_t \Delta_t \right\} 
-\frac{\tau \mu_t}{2}\left(\Delta_t\right)^2  
+ \tau M_{t-1}\Delta_{t-1} \\
&\stackrel{\eqref{lipschG}}
\le  \left(1-\tau\right) \left \{
  e_t^+ +
 2{M}_t \Delta_t \right\} 
-\frac{\tau \mu_t}{2}\left(\Delta_t\right)^2  
+ 2 \tau M_{t-1}\Delta_{t-1}
\end{align*}
where the last inequality follows from \eqref{ineq:partial}.
\end{proof}

The next result gives bounds on the quantities
$e_t^{k+1}$
defined in \eqref{defetk} in terms of the $\Delta_t^k$'s.

\begin{lemma} \label{lm:error-recur}
For every $t=1,\ldots,T-1$ and $k \ge 1$, we have   \begin{equation}\label{bdetk}
    0 \leq e_t^{k+1} \leq 2 \sum_{s=t+1}^{T-1} {M}_s \Delta_s^k.
    \end{equation}

\end{lemma}

\begin{proof} We fix $k \geq 1$ and
use again the shorthand notation
\eqref{shorthand}. The first inequality in \eqref{bdetk} immediately follows from the definition of $e_t$ in \eqref{defetk} and the fact that $G_t(x_{\ge t+1};x_t)$ and $m_t$ are upper and lower bounds on $\Q_t(x_t)$.

First, observe that
$e_{T-1}^+=G_{T-1}(x_t^+,x_{t-1}^+)-m^+_{T-1}=F_T(x_t^+,x_{t-1}^+)-m^+_{T-1}=0$ due
to the definitions of $e_{T-1}^+$ and $G_{T-1}$ in \eqref{defetk} and \eqref{def:Gt}, respectively, the identities in \eqref{def:vtk} with $t=T$,
and the fact that $\Gamma_T^+$ is the null function.
We next show that
for every $t \in \{2,\ldots,T-1\}$, we have
    \begin{equation} \label{bdetk0}
    e_{t-1}^+ - e_t^+ \le 2 M_t \Delta_t.
    \end{equation}
    Using \eqref{def:Gt}, \eqref{def:vtk}, \eqref{ineq:require}, \eqref{eq:sim-cond1'}, and \eqref{defetk}, we conclude that for any $t \in \{2,\ldots, T-1\}$,
\begin{align*}
e_{t-1}^+ & \stackrel{\ \ \, \eqref{defetk} \ \ \, }= G_{t-1}\left(x^+_{\ge t};x^+_{t-1}\right)-\m_{t-1}^+  \\
&\stackrel{\eqref{def:vtk},\eqref{def:Gt}}
= 
F_t\left(x_t^+,x_{t-1}^+\right)+
 G_t\left(x^+_{\ge t+1};x^+_t\right) - \left[F_t\left(x_t^+,x_{t-1}^+\right)+\Gamma_t^+\left(x_t^+\right)\right]\\
& \stackrel{\qquad \ \ \ }=  \left[ G_t\left(x^+_{\ge t+1};x^+_t\right)-m_t^+ \right] + \left[m_t^+-\Gamma_t^+\left(x_t^+\right) \right] \\
&\stackrel{\ \ \, \eqref{defetk}\ \ \, }{=} e_t^++ \left[ m_t^+-\Gamma_t^+\left(x_t^+\right) \right]  \stackrel{\ \eqref{ineq:require}\ }\leq e_t^++ m_t^+-\ell_t^+\left(x_t^+\right) \\
& \stackrel{\ \ \, \eqref{eq:sim-cond1'} \ \ \, }\leq e_t^+ + 2 {M}_t \Delta_t.
\end{align*}
 We have thus proved that \eqref{bdetk0} holds.
The above conclusions then imply that,
for every $t\in \{1,\ldots,T-2\}$,
\[
e_t^+ =  e_t^+ - e_{T-1}^+ = \sum_{s=t+1}^{T-1} \left(e_{s-1}^+-e_{s}^+\right) \le 2\sum_{s=t+1}^{T-1} 
{M}_s \Delta_s,
\]
and hence that the second inequality in \eqref{bdetk}  holds.
\end{proof}

The following lemma gives a generic recursive formula for the aggregation of gaps $\{p_t^k\}$ over stages with general $\tau \in (0,1)$ and
$\theta=(\theta_1,\ldots,\theta_t) \in \R^{T}_{+}$.

\begin{lemma}\label{lem:tk}
For any $\tau \in (0,1)$ and
$\theta=(\theta_1,\ldots,\theta_{T-1}) \in \R^{T-1}_{+}$ such that
$\theta_1=1$ and $\theta_T=0$, the sequence $\{r^k(\theta)\}_{k \ge 1}$ defined as
\begin{equation}\label{def:tk}
        r^k(\theta):= \sum_{t=1}^{T-1} \theta_t p_{t-1}^k \qquad \forall \, k \ge 1
    \end{equation}
satisfies
    \begin{equation}\label{ineqrk}
    r^{k+1}(\theta) \le \tau r^k(\theta) +
(1-\tau) \sum_{t=1}^{T-1} C_t(\tau,\theta) \qquad \forall \, k \ge 1,
    \end{equation}
where
\begin{equation}\label{def:ct}
    C_t(\tau,\theta):=
 \frac{2 M_t^2 \left[(1-\tau)\sum_{i=1}^{t} \theta_i+\tau \theta_{t+1}\right]^2}{\tau (1-\tau)\mu_t \theta_t}, \quad \forall \, t=1,\ldots,T-1.
\end{equation}
\end{lemma}

\begin{proof}
Again, we employ the notation \eqref{shorthand}, along with $(r^+, r) = (r^{k+1}(\theta), r^k(\theta)).$ It follows from Lemmas~\ref{ptk} and \ref{lm:error-recur}(b) that for $t=1,\ldots,T-1$,
\begin{align}
p_{t-1}^+-\tau p_{t-1} + \frac{\tau \mu_t}{2}(\Delta_t)^2
& \stackrel{\eqref{p1ineq}}\leq  
(1-\tau) \left[e_t^+ + 2M_t \Delta_t \right]  +  2\tau M_{t-1} \Delta_{t-1} \nn \\
 & \stackrel{\eqref{bdetk}}{\leq}  
2(1-\tau)\sum_{i=t}^{T-1}  {M}_i \Delta_i +  2\tau M_{t-1} \Delta_{t-1}. \label{ineq:e2}
\end{align}
Multiplying inequality \eqref{ineq:e2} 
 by $\theta_t$,
summing 
the resulting inequality from $t=1$ to $T-1$, and using \eqref{def:tk}, we have
\begin{equation}\label{ineqr}
\begin{aligned}
    r^+-\tau r + \frac{\tau}{2}\sum_{t=1}^{T-1}\mu_t\theta_t(\Delta_t)^2
&\le 2(1-\tau)\sum_{t=1}^{T-1}\theta_t\sum_{i=t}^{T-1} M_i\Delta_i+2 \tau\sum_{t=1}^{T-1}\theta_t M_{t-1}\Delta_{t-1}\\
&=\sum_{t=1}^{T-1} M_t\Delta_t\left[2(1-\tau)\sum_{i=1}^{t}\theta_i+2\tau\theta_{t+1}\right]
\end{aligned}
\end{equation}
where the last equality follows from facts that
$\Delta_0=0$, $\theta_t=0$,
and the identities 
\[
\sum_{t=1}^{T-1}\theta_t\sum_{i=t}^{T-1} M_i\Delta_i
= \sum_{t=1}^{T-1} M_t\Delta_t \sum_{i=1}^{t}\theta_i,
\qquad \sum_{t=1}^{T-1}\theta_t M_{t-1}\Delta_{t-1} = \sum_{t=1}^{T-1}\theta_{t+1} M_t\Delta_t.
\]
Inequality
\eqref{ineqrk} now follows from inequality \eqref{ineqr}
the inequality $-a_t x^2+b_t x\le b_t^2/(4a_t)$ with
\[
a_t=\frac{\tau}{2}\mu_t\theta_t,\qquad
b_t=2 M_t\left[(1-\tau)\sum_{i=1}^{t}\theta_i+\tau\theta_{t+1}\right],\qquad
x=\Delta_t,
\]
and the definition of $C_t(\tau,\theta)$ in \eqref{def:ct}.
\end{proof}

The following lemma is a special case of Lemma \ref{lem:tk} with specific $\theta$ and $\tau$, which are of interest in our analysis of FDDP.

\begin{lemma}\label{lem:new-r}
Define $\overline Q$ by \eqref{defQeps}, and
\begin{equation}\label{defQeps2}
\varepsilon := \min \left\{ \frac{1}{2} ,  \frac{\bar \varepsilon}{\overline Q}  \right\}, \quad \tau := \frac{1}{1+\varepsilon^{2^{T-1}-1} }\in (0,1),
\end{equation}
    where $\bar \varepsilon>0$ is the tolerance input to FDDP.
Moreover, define $\{r^k\}_{k \ge 1}$ as the sequence $\{r^k(\theta)\}_{k \ge 1}$ defined in Lemma \ref{lem:tk}  with
weights $\theta_t$ given by
\begin{equation}\label{def:theta}
    \left\{
\begin{aligned}
\theta_t &= \varepsilon^{\,2^{T-1}-2^{T-t}}, \quad t=1,2,\ldots,T-1,\\
\theta_T &= 0.
\end{aligned}
\right.
\end{equation}
Then, we have for every $t=1,\ldots, T-1$, $\sum_{i=1}^t \theta_i \le 2$, and for every $k \ge 1$,
\begin{equation}\label{ineq:tk-T}
    r^{k+1} \le \tau r^k+(1-\tau) \frac{\bar \varepsilon}{2}.
\end{equation}
\end{lemma}

\begin{proof} 
Observe that for $i=2, \ldots,T-1$, we have
\begin{equation}\label{ineq1sp}
2^{T-2}-2^{T-i}=2^{T-i}\left(2^{i-2}-1\right)
\geq 2^{i-2}-1\geq i-2.
\end{equation}
This inequality and the definition of $\theta_t$ in \eqref{def:theta} imply that for every $t=1,\ldots, T-1$,
\[
\sum_{i=1}^{t}\theta_i
=1+\varepsilon^{2^{T-2}}\sum_{i=2}^{t}\varepsilon^{\,2^{T-2}-2^{T-i}}
\stackrel{\eqref{ineq1sp}}\le 1+\varepsilon^{2^{T-2}}\sum_{i=2}^{t}\varepsilon^{\,i-2}
\le 1+\frac{\varepsilon^{2^{T-2}}}{1-\varepsilon}
\le 1+\frac{\varepsilon}{1-\varepsilon}\le 2,
\]
where the last inequality is due to the fact that $\varepsilon \leq 1/2$.
Using the above inequality and the definition of $C_t(\tau,\theta)$ in \eqref{def:ct} with $\tau$ and $\theta$ in \eqref{defQeps2} and \eqref{def:theta}, respectively, we have
\begin{equation}\label{ineq:Ct}
    C_t(\tau,\theta) \stackrel{\eqref{def:ct}}{\le} \frac{2 M_t^2 \left[2(1-\tau) + \tau \theta_{t+1}\right]^2}{\tau (1-\tau)\mu_t \theta_t}
\le \frac{M_t^2}{\mu_t} \left( \frac{16 \, (1-\tau)}{\tau \theta_t} + \frac{4\tau \theta_{t+1}^2}{(1-\tau) \theta_t} \right),
\end{equation}
where the second inequality follows from the fact that $(a+b)^2\le 2a^2+2b^2$ for every $a,b \in \R$.
Moreover, the definitions of $\tau$ and $\theta_t$ in \eqref{defQeps2} and \eqref{def:theta}, respectively, imply that for $t=1,\ldots,T-1$,
\[
\frac{1-\tau}{\tau \theta_t}=\varepsilon^{2^{T-t}-1}\le \varepsilon,
\qquad
\frac{\tau \theta_{t+1}^2}{(1-\tau)\theta_t} \le \varepsilon.
\]
Plugging the above two inequalities into \eqref{ineq:Ct} yields that for $t=1,\ldots,T-1$,
\[
C_t(\tau,\theta) \le \frac{20  M_t^2 \varepsilon}{\mu_t}.
\]
Using the above inequality, the definition of $\overline Q$ in \eqref{defQeps}, and Lemma \ref{lem:tk}, we conclude that
\[
r^{k+1} \stackrel{\eqref{ineqrk},\eqref{defQeps}}\leq \tau r^k + 
(1-\tau)\frac{\overline Q \varepsilon}{2}.
\]
Finally, \eqref{ineq:tk-T} immediately follows from the definition of $\varepsilon$ in \eqref{defQeps2}.
\end{proof}

Before proving Theorem
\ref{prop:null-strong}, we need the following lemma
which shows that
$\overline r$ given by \eqref{def:barr} is an upper bound
for $r^1$ defined in Lemma \ref{lem:new-r}. 

\begin{lemma}\label{ubr1}
The scalar $r^1$ defined in Lemma \ref{lem:new-r} satisfies $r^1 \le \overline r$ where $\overline r$ is as in \eqref{def:barr}.
\end{lemma}

\begin{proof}
It follows from the definition of $r^1(\theta)$ in \eqref{def:tk} with $\theta$ as in \eqref{def:theta} that
\begin{equation}\label{eq:r1}
    r^1 = r^1(\theta) \stackrel{\eqref{def:tk}}= \sum_{t=1}^{T-1} \theta_t p_{t-1}^1 \stackrel{\eqref{def:ptk}} = \sum_{t=1}^{T-1} \theta_t \left( G_{t-1}\left(y^1[t]; x^{1}_{t-1}\right) - m_{t-1}^1 \right),
\end{equation}
where $p_{t-1}^1$ is as in \eqref{def:ptk}.
Using \eqref{ineq:obs-G} with $k=1$ and the definitions of $G_{t-1}$ and $m_{t-1}^1$ in \eqref{def:Gt} and \eqref{def:vtk}, respectively, we have
\begin{align*}
& G_{t-1}\left(y^1[t]\,; x^{1}_{t-1}\right) - m_{t-1}^1 
 \stackrel{\eqref{ineq:obs-G}}\le G_{t-1}\left(x^1_{\ge t}\,; x^{1}_{t-1}\right) - m_{t-1}^1  \\
&\stackrel{\eqref{def:Gt},\eqref{def:vtk}}= \sum_{s=t}^T F_s\left(x_s^1, x_{s-1}^1\right) - F_t\left(x_t^1, x_{t-1}^1\right) - \Gamma_t^{1}\left(x_t^1\right) 
= \sum_{s=t+1}^T F_s\left(x_s^1, x_{s-1}^1\right) - \Gamma_t^{1}\left(x_t^1\right).
\end{align*}
Substituting the above bound back into the definition of $r^1$ in \eqref{eq:r1}, we obtain
\begin{align*}
r^1 &\stackrel{\eqref{eq:r1}}\le \sum_{t=1}^{T-1} \theta_t \left( \sum_{s=t+1}^T F_s\left(x_s^1, x_{s-1}^1\right) - \Gamma_t^{1}\left(x_t^1\right) \right) \\
    &= \sum_{s=2}^T F_s\left(x_s^1, x_{s-1}^1\right) \left( \sum_{t=1}^{s-1} \theta_t \right) - \sum_{t=1}^{T-1} \theta_t \Gamma_t^{1}\left(x_t^1\right) \le \sum_{s=2}^T U_s \left( \sum_{t=1}^{s-1} \theta_t \right) - \sum_{t=1}^{T-1} \theta_t L_t,
\end{align*}
where the last inequality follows from the assumption in Theorem \ref{prop:null-strong} that $U_s$ (resp., $L_t$) is an upper (resp. lower) bound on $F_s$ (resp., $\Gamma_t$)
Furthermore, we recall from \eqref{def:theta} that $\theta_t \le 1$, and from Lemma \eqref{lem:new-r} that $\sum_{t=1}^{s-1} \theta_t \le 2$. Applying these coefficient bounds and using the definition of $\overline r$ in \eqref{def:barr} and the fact that $L_t <0$ gives $r^1 \le \overline r$ as claimed.
\end{proof}

$\vspace*{0.1cm}$

We are now ready to prove Theorem \ref{prop:null-strong}.
\vgap

\noindent
{\textbf{Proof of Theorem \ref{prop:null-strong}.}}
    Defining $\tau$ by 
\eqref{defQeps2} with $\varepsilon$ given
in \eqref{defQeps}, using the inequality $\tau \le e^{\tau-1}$
	    and  Lemma \ref{lem:new-r}, we  conclude that 
        \[
		r^k - \frac{\bar \varepsilon}{2} \stackrel{\eqref{ineq:tk-T}}\le \tau^{k-1}\left(r^1 - \frac{\bar \varepsilon}{2} \right) \le \tau^{k-1}r^1 \le e^{(\tau-1) (k-1)} r^1. 
		\]
        To have $r^k \le \bar \varepsilon$, we need at most        
        \begin{align}
        \frac{1}{1-\tau} \log\left( \frac{2 r^1}{\bar \varepsilon}\right) + 1
            &=\left(1+\frac{1}{\varepsilon^{2^{T-1}-1}}\right)\log\left( \frac{2 r^1}{\bar \varepsilon}\right) + 1 \\
             &\leq  \left(1+\left[\max\left(2,\frac{\overline Q}{\bar \varepsilon}\right)\right]^{2^{T-1}-1} \right)\log\left( \frac{2 {\overline r}}{\bar \varepsilon}\right) + 1 
        \end{align}
        iterations, where we have used $r^1 \leq \overline r$ (from Lemma \ref{ubr1}).

\section{Numerical experiments}\label{numsim}

\par {\textbf{Problem and goal.}}
For simplicity, this section focuses on problem \eqref{pb:multistage0T} instead of the main problem \eqref{pb:multistage0Tc} of the paper, since the core of our development, FDDP, is designed to solve the unconstrained problem \eqref{pb:multistage0T}.
We illustrate that both multi-cut and two-cut DDP variants can be more efficient than 
a direct solution method to solve \eqref{pb:multistage0T} presenting
the CPU time
versus problem parameters $n$, $T$, and $\mu_t$.
The experiments are done with 
the following optimization problem
\begin{equation}\label{numproblem}
\left\{
\begin{array}{@{}l@{\quad}l@{}}
\displaystyle \min_{x_1,\ldots,x_T \in \mathbb{R}^n} &
\displaystyle \sum_{t=1}^T \max_{i=1,\ldots,m}
\Bigg[
\frac12
\begin{pmatrix} x_{t-1} \\ x_t \end{pmatrix}^{\!\top}
\left(\xi_{ti}\xi_{ti}^\top + \lambda_i I_{2n}\right)
\begin{pmatrix} x_{t-1} \\ x_t \end{pmatrix}
+\xi_{ti}^\top
\begin{pmatrix} x_{t-1} \\ x_t \end{pmatrix}
\Bigg] \\[2pt]
\text{\, \quad s.t.} &
x_t \ge 0,\qquad \mathbf{1}^\top x_t = 1,\qquad t=1,\ldots,T
\end{array}
\right.
\end{equation}
for known $x_0$, $\lambda_i>0$, and parameters $\xi_{t i} \in \mathbb{R}^{2n}$. This problem is of form \eqref{pb:multistage0T}
with
$$F_t(x_t,x_{t-1})=\max_{i=1,\ldots,m} f_{t i}(x_t,x_{t-1})$$ for
\[
f_{t i}(x_t,x_{t-1}) = \frac12
\begin{pmatrix} x_{t-1} \\ x_t \end{pmatrix}^{\!\top}
\left(\xi_{ti}\xi_{ti}^\top + \lambda_i I_{2n}\right)
\begin{pmatrix} x_{t-1} \\ x_t \end{pmatrix}
+\xi_{ti}^\top
\begin{pmatrix} x_{t-1} \\ x_t \end{pmatrix}.
\]

It is easy to check that Assumptions (A0)-(A3) of our complexity analysis from Section \ref{sec:analysis} are satisfied.
In particular, functions
$F_t(\cdot,x_{t-1})$ are $\mu_t$-strongly convex
with $\mu_t=\min_{i=1,\ldots,m} \lambda_i$.

\if{
\red{??}
\begin{lemma}
Let $f_i:X \rightarrow \mathbb{R}, i \in I$, be convex functions
defined on a convex set $X$ with $f_i$ being $\alpha_i$-strongly convex. Then
$f(x)=\max_{i \in I} f_i(x)$ is 
$\mu$-strongly convex with 
\begin{equation}
\mu=\displaystyle \min_{i \in I} \; \alpha_i.
\end{equation}
\end{lemma}
\begin{proof} For every $x,y \in X$ and $t \in [0,1]$, we have
for every $i \in I$ that
$$
\begin{array}{lcl}
f_i(t x+(1-t)y) & \leq &\displaystyle   t f_i(x)+(1-t)f_i(y) - \frac{t(1-t) \alpha_i}{2}\|x-y\|^2 \\
& \leq & \displaystyle  t f(x)+(1-t)f(y) - \frac{t(1-t) \mu}{2}\|x-y\|^2 \\
\end{array}
$$
and therefore
$$
f(t x+(1-t)y)  \ \leq \ \displaystyle  t f(x)+(1-t)f(y) - \frac{t(1-t) \mu}{2}\|x-y\|^2 
$$
which achieves the proof.
\end{proof}
}\fi

We can write dynamic programming
equations for problem \eqref{numproblem} and apply two-cut DDP (a special instance of FDDP) and multi-cut DDP
to solve it.

\if{
We now describe for our instance multi-cut and 2-cuts DDP.

\par {\textbf{Multi-cut DDP.}} In multi-cut DDP, at iteration $k$, we approximate in these equations
$\Q_t$ for $t=1,\ldots,T$, by
$\Gamma_t^{k-1}$, knowing that we start the method with
known affine lower-bounding functions
$\Gamma_t^0(x_t)=\langle a_t^0, x_t \rangle + b_t^0$ and that both  $\Gamma_T^k$ and $\Q_T$ are null
functions for all $k$. We also have that
$\Gamma_t^{k-1}(x_t)=\max_{\ell=0,\ldots,k-1} \langle a_t^\ell, x_t \rangle + b_t^\ell$
is a maximum of affine functions.
We now give the details of iteration $k$.
In the forward pass, we compute for $t=1,\ldots,T$,
and given $x_{t-1}^k$ (with $x_0^k=x_0$ given)
an optimal solution $x_t^k$ of
\begin{equation}\label{optforward}
\begin{array}{l}
\left\{
\begin{array}{l}
\displaystyle \min_{x_t} F_t\left(x_t,x_{t-1}^k\right)+\Gamma_t^{k-1}(x_t)\\
x_t \geq 0,\;\sum_{i=1}^n x_t(i)=1,
\end{array}
\right.
=\left\{
\begin{array}{l}
\displaystyle \min_{x_t,\theta_1,\theta_2} \theta_1+ \theta_2\\
x_t \geq 0,\;\sum_{i=1}^n x_t(i)=1,\\
\theta_1 \geq f_{ti}\left(x_t,x_{t-1}^k\right),\;i=1,\ldots,m,\\
\theta_2  \geq \left\langle a_t^{\ell},x_t \right\rangle + b_t^{\ell},\;\ell=0,\ldots,k-1,
\end{array}
\right.\\
\red{??} \geq  \left\{
\begin{array}{l}
\displaystyle \min_{x_t,\theta_1,\theta_2} \theta_1+ \theta_2\\
x_t \geq 0,\;\sum_{i=1}^n x_t(i)=1,\\
\theta_1 \geq 0.5\left(x_{t-1}^k\right)^T A_{t i} x_{t-1}^k+x_t^T B_{t i}^T x_{t-1}^k + 0.5 x_t^T C_{t i} x_t + \xi_{ti}^T\begin{pmatrix} x_{t-1}^k \\ x_t \end{pmatrix},\;\forall \  \ i,\\
\theta_2  \geq \left\langle a_t^{\ell},x_t \right\rangle + b_t^{\ell},\;\ell=0,\ldots,k-1,
\end{array}
\right.
\end{array}
\end{equation}
where 
$$
A_{t i}=\xi_{ti}(1:n)\xi_{ti}(1:n)^T+\lambda_i I_n,
$$
$$
B_{t i}=\xi_{ti}(1:n)\xi_{ti}(n+1:2n)^T,
$$
and
$$
C_{t i}=\xi_{ti}(n+1:2n)\xi_{ti}(n+1:2n)^T+\lambda_i I_n.
$$
We compute the upper bound $p_0^k$ on the optimal value.
We then update the approximations on $\mathcal{Q}_t$
backward from $t=T$ to $t=1$.
We start with $\Gamma_T^k$ the null function meaning
that $a_T^k$ and $b_T^k$ are null.
Then given $\Gamma_{t+1}^k$, for $t \in \{1,\ldots,T-1\}$, we compute
\begin{equation}\label{optforward2}
\red{??}\mathcal{V}_t^k\left(x_t^k\right)
=\left\{
\begin{array}{l}
\displaystyle \min_{x_{t+1},\theta_1,\theta_2} \theta_1+ \theta_2\\
x_{t+1} \geq 0,\;\sum_{i=1}^n x_{t+1}(i)=1,\\
\theta_1 \geq f_{t+1i}\left(x_{t+1},x_t^k\right),\;i=1,\ldots,m,\;\;\;\;\;\;\;[\lambda_{t+1i}]\\
\theta_2  \geq \left\langle a_{t+1}^{\ell},x_{t+1} \right\rangle + b_{t+1}^{\ell},\;\ell=0,\ldots,k.
\end{array}
\right.
\end{equation}
Denoting by $x_{t+1}^{\star k}$ an optimal solution
of \eqref{optforward2}
and by $\lambda_{t+1 i}$ an optimal Lagrange multiplier associated with constraint
$\theta_1 \geq f_{t+1i}(x_{t+1},x_t^k)$, we have
$$
\Q_t(x_t) \geq 
\mathcal{V}_t^k(x_t) \geq 
\mathcal{V}_t^k\left(x_t^k\right)
+ \left\langle s_t^k , x_t-x_t^k \right\rangle
$$
where
$s_t^k=\displaystyle \sum_{i=1}^m \lambda_{t+1 i}\nabla_{x_t} f_{t+1 i}(x_{t+1}^{\star k},x_t^k)$ is given by
$$ 
\sum_{i=1}^m \lambda_{t+1 i}\left(
\xi_{t+1 i}(1:n)+
\left(\xi_{t+1 i}(1:n)\xi_{t+1 i}(1:n)^T+\lambda_i I_n \right)x_t^k+2 \xi_{t+1 i}(1:n)\xi_{t+1 i}(n+1:2n)^T x_{t+1}^{\star k}
\right), 
$$
meaning that we have
$$
a_t^k=s_t^k,
\quad 
b_t^k=\mathcal{V}_t^k\left(x_t^k\right)-\left\langle s_t^k ,x_t^k \right\rangle.
$$
We then compute
 the lower bound
$m_0^k$ on the optimal value 
of \eqref{numproblem}
which is the optimal value of \eqref{optforward2} for
$t=0$. 

\par {\textbf{Two-cut DDP.}}
We now show how the FDDP framework can be used to construct a two-cut lower approximation of $\Q_t$, which gives rise to two-cut DDP.

We initialize with
\[
\Gamma_t^0(x_t)=\max\left(\left \langle a_t^{0 1}, x_t \right \rangle + b_t^{0 1},\,\left \langle a_t^{0 2}, x_t \right \rangle + b_t^{0 2}\right),
\]
where
\[
a_t^{0 2}=a_t^{0 1}=a_t^{0},
\qquad
b_t^{0 2}=b_t^{0 1}=b_t^{0},
\]
so that the initial lower approximation is simply the affine lower-bounding function
\[
\left \langle a_t^0, x_t \right \rangle + b_t^0.
\]
We also note that, for every $k$, both $\Gamma_T^k$ and $\Q_T$ are the null function.

At iteration $k$ and stage $t$, the function $\Q_t$ is approximated in the forward pass by
\[
\Gamma_t^k(x_t)=\max\left(\left \langle a_t^{k-1,1}, x_t \right \rangle + b_t^{k-1,1},\,\left \langle a_t^{k-1,2}, x_t \right \rangle + b_t^{k-1,2}\right).
\]
The trial points $x_t^k$ are obtained by performing the forward pass with $\Q_t$ replaced by $\Gamma_t^k$. Two cuts are then computed for each function $\Q_t$ as follows. Recalling that
\begin{equation}\label{firststodpa2Tbm}
\V_t^k(x_t)=
 \left\{
\begin{array}{l}
\displaystyle \min_{u_{t+1}\in \mathbb{R}^n} F_{t+1}(u_{t+1},x_t) +\Gamma_{t+1}^k( u_{t+1} )\\
u_{t+1} \in X_{t+1},
\end{array}
\right.
\end{equation}
we compute
\begin{equation}\label{firststodpa2Tbmt}
\V_t^k\left(x_t^k\right)=
 \left\{
\begin{array}{l}
\displaystyle \min_{x_{t+1},c} F_{t+1}\left(x_{t+1},x_t^k\right) + c \\
x_{t+1} \geq 0,\;\sum_{i=1}^n x_{t+1}(i)=1,\\
c \geq A_{tk1}(x_{t+1}), \;\;[\gamma_{tk}],\\
c \geq A_{tk2}(x_{t+1}),\;\;[1-\gamma_{tk}],
\end{array}
\right.
\end{equation}
where $A_{tk1}$ and $A_{tk2}$ denote the two affine functions defining $\Gamma_{t+1}^k$, and where $\gamma_{tk}$ and $1-\gamma_{tk}$ are the optimal Lagrange multipliers associated with the constraints
\[
c \geq A_{tk1}(x_{t+1})
\qquad\text{and}\qquad
c \geq A_{tk2}(x_{t+1}),
\]
respectively.

The new lower approximation of $\Q_{t+1}$ is then defined by
\[
\Gamma_{t+1}^{k+1}(\cdot)
=
\max\left(
\gamma_{tk}A_{tk1}(\cdot)+(1-\gamma_{tk})A_{tk2}(\cdot),\,
\ell_{\V_{t+1}^k}\left(\cdot;x_{t+1}^k\right)
\right),
\]
that is, by choosing the two-cut black box. This yields the new cut coefficients
\[
a_{t+1}^{k1},\; a_{t+1}^{k2},\; b_{t+1}^{k1},\; b_{t+1}^{k2}.
\]

It is easy to check that $\Gamma_{t+1}^{k+1}$ remains, like $\Gamma_{t+1}^k$, a lower approximation of $\Q_{t+1}$. Indeed, the relation $\Gamma_{t+2}^k \leq \Q_{t+2}$ implies
\[
\Q_{t+1}(x_{t+1}) \geq \V_{t+1}^k(x_{t+1}),
\]
and therefore
\[
\Q_{t+1}(x_{t+1}) \geq \V_{t+1}^k(x_{t+1})
\geq \ell_{\V_{t+1}^k}\left(x_{t+1};x_{t+1}^k\right).
\]
Moreover, since
\[
\Q_{t+1} \geq \max(A_{tk1},A_{tk2}),
\]
we also have
\[
\Q_{t+1} \geq \gamma_{tk}A_{tk1}+(1-\gamma_{tk})A_{tk2}.
\]
It follows that $\Gamma_{t+1}^{k+1}$ is again a lower approximation of $\Q_{t+1}$.

}\fi

\noindent {\textbf{Results.}} Our implementation was coded in Matlab and Gurobi was used to solve linear and quadratic programs. We compare multi-cut DDP, two-cut DDP with DSM (Direct Solution Method) to solve
several instances of \eqref{numproblem} where DSM
solves the problem pushing the objective in the constraints as
\begin{equation}\label{reformpb}
\begin{array}{l}
\displaystyle \min_{x_1,\ldots,x_T,\theta_1,\ldots,\theta_T} \sum_{t=1}^T \theta_t\\
x_t \geq 0,\;\displaystyle \sum_{i=1}^n x_t(i)=1,\\
\theta_t \geq f_{ti}(x_t,x_{t-1}), \ \forall \ t=1,\ldots,T, \ \forall \ i=1,\ldots,m,
\end{array}
\end{equation}
and calling Gurobi quadratic solver to solve \eqref{reformpb}. 
We first take $\lambda_i=100$ for all $i$
and $m=2$. We then consider three values of
$T$: $2,5,10$, and the following 13 values of
$n$: 10, 20, 50, 100, 200, 300, 400, 500, 600, 700,
800, 900, and 1000. We report in Figure \ref{empprobmodel1} the evolution of 
multi-cut DDP, two-cut DDP, and DSM CPU time as a function of state vector dimension $n$ for $T= 2$ (upper left plot), $T= 5$ (upper right plot),  and $T= 10$ (bottom plot). We observe that for this strongly convex problem,
multi-cut DDP and two cut DDP CPU time are similar and increase with a low rate, as a function of $n$. 
This is in accordance with our Theorem \ref{prop:null-strong} which indicates
that DDP does not suffer from the curse of dimensionality
when $n$ increases for strongly convex problems, i.e., 
CPU time in practice slowly increased with $n$ on our instances. DSM is much slower with a CPU time that increases with a larger rate as a function of state vector dimension $n$. 
We stopped DDP when a solution with relative accuracy
5\% was found. Correctness of the implementation was checked observing very close values between the approximate objective with DSM, and the upper and lower bounds at termination of both DDP methods.

\begin{figure}
    \centering
    \begin{tabular}{cc}
        \includegraphics[scale=0.6]{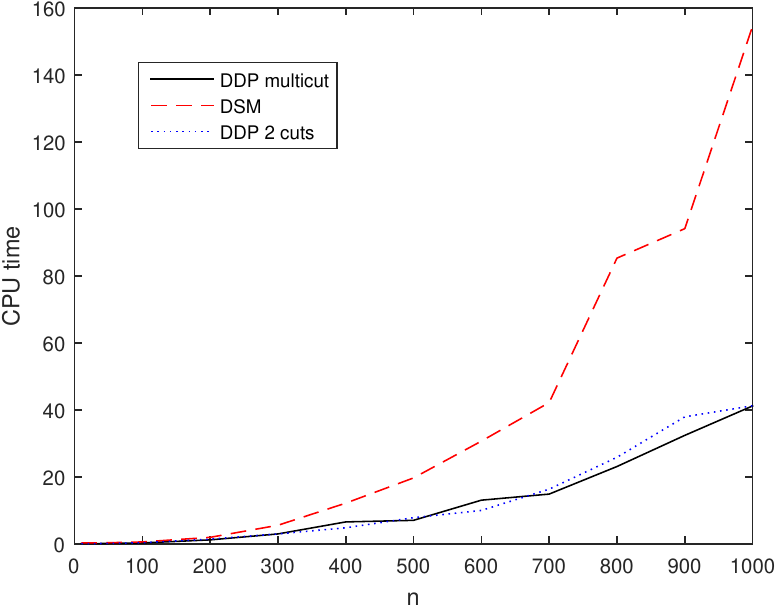}&
        \includegraphics[scale=0.6]{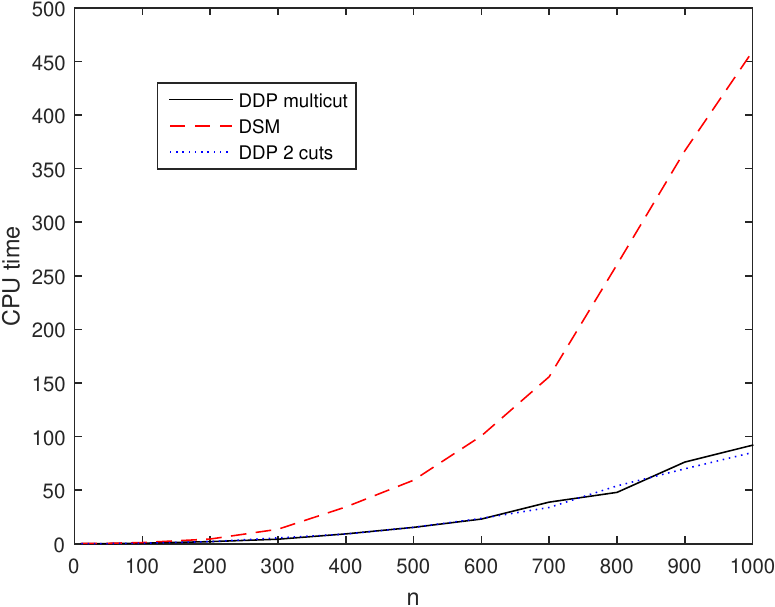}
     \end{tabular}
\begin{tabular}{c}
        \includegraphics[scale=0.6]{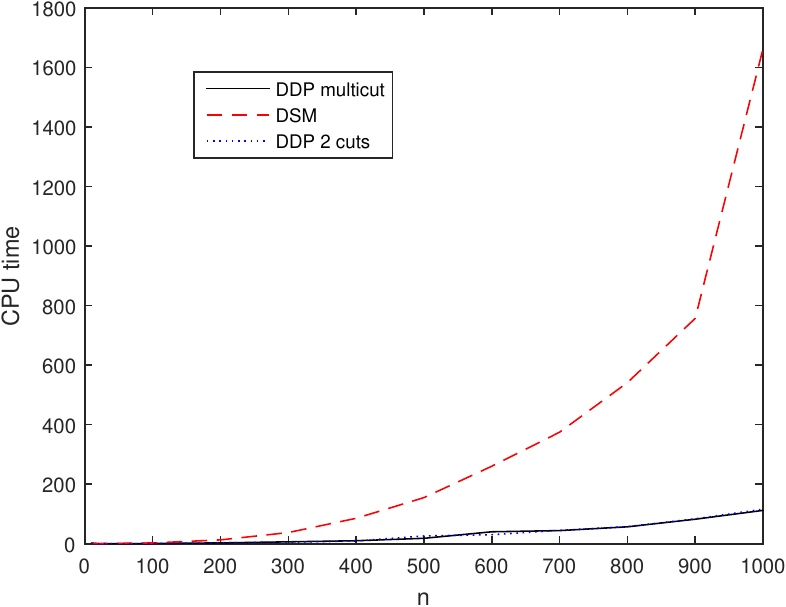}
     \end{tabular}
 \caption{Evolution of multi-cut DDP, two-cut DDP, and DSM CPU time as a function of state vector dimension $n$ for $m=2$, $\lambda_i=100$, $T= 2$ (upper left figure), $T= 5$ (upper right figure),  and $T= 10$ (bottom figure).}
\label{empprobmodel1}
\end{figure}

We now fix $m=2$, $T= 5$, $n=1000$ and let $\lambda_i$ vary in the set $\{0.1,0.5,1,10,100\}$. The evolution
of CPU time for multi-cut DDP, two-cut DDP, and DSM
as a function of the constant of strong convexity
$\lambda_i$ is given in Figure \ref{empprobmodel2}.
As forecast by our Theorem \ref{prop:null-strong},
CPU time to get an approximate solution decreases (here slowly) when the constant of strong convexity
$\lambda_i$ increases. 

\begin{figure}
    \centering
\begin{tabular}{c}
        \includegraphics[scale=0.6]{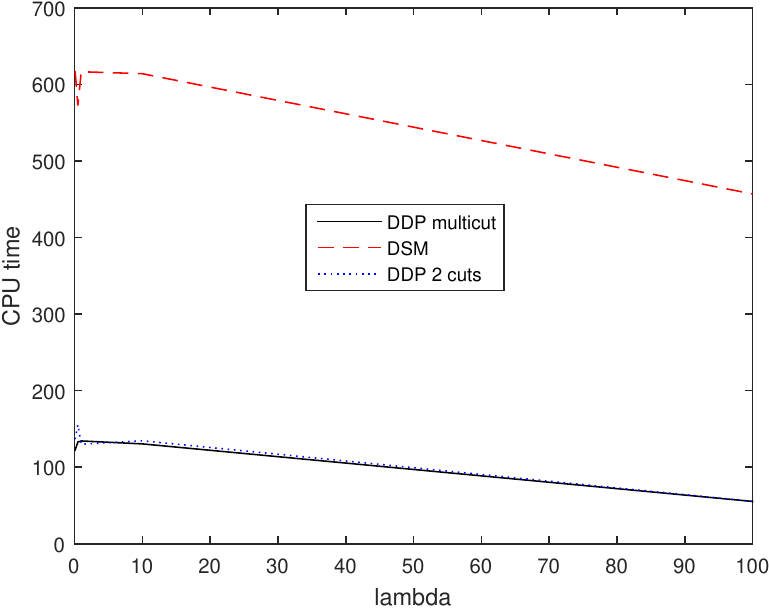}
     \end{tabular}
\caption{Evolution of multi-cut DDP, two-cut DDP, and DSM
CPU time
for $\lambda$ varying in the set $\{0.1,0.5,1,10,100\}$
 and for $m=2$, $T= 5$, and $n=1000$.}
\label{empprobmodel2}
\end{figure}

Finally, we fix $m=2$, $\lambda_i=100$, $n=10$ and let $T$ vary in the set $\{2,5,10,100,200,500,1000\}$. The evolution
of CPU time for multi-cut DDP, two-cut DDP, and DSM
as a function of $T$ is given in Figure \ref{empprobmodel3}.
As forecast by our Theorem \ref{prop:null-strong},
CPU time increases with a high rate with $T$ for DDP. However,
the increase is linear in these instances whereas
in theory it can be an exponential function of $T$.
For all three methods the increase is a similar function
of $T$.

\begin{figure}
    \centering
\begin{tabular}{c}
        \includegraphics[scale=0.6]{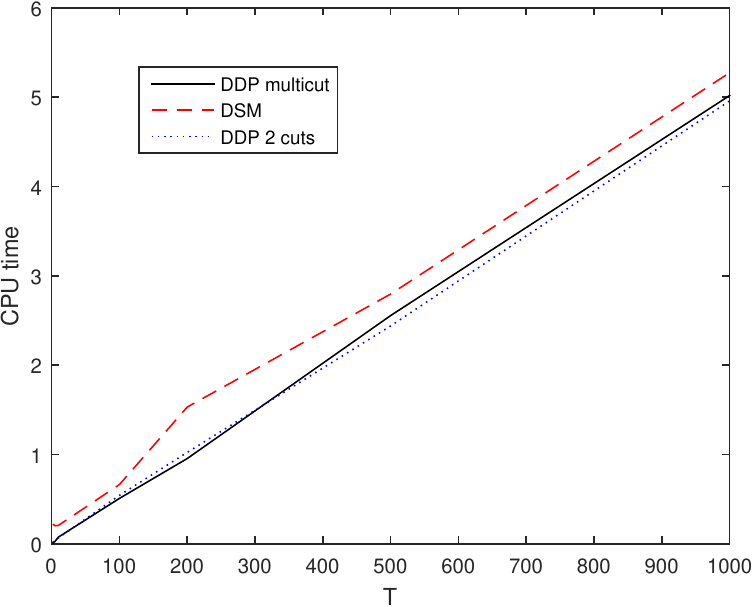}
     \end{tabular}
\caption{Evolution of multi-cut DDP, two-cut DDP, and DSM
CPU time
for $T$ varying in the set $\{2,5,10,100,200,500,1000\}$
 and for $m=2$, $n= 10$, and $\lambda_i=100$.}
\label{empprobmodel3}
\end{figure}

\section{Conclusion}

This paper introduces FDDP, a flexible dual dynamic programming framework for solving a class of convex optimization problems with linear coupling constraints. FDDP covers classical multi-cut DDP and the proposed two-cut DDP variant, which maintains only two affine cuts per iteration.
Our main result establishes an iteration-complexity bound for FDDP that is independent of the dimensions of the state vectors. The paper also shows how the complexity applies to constrained  convex problems via a penalized and perturbed reformulation.

A natural next step is to extend the present dimension-free analysis of FDDP to SDDP for solving multistage stochastic optimization problems. In the convex case, complexity results for SDDP were established in \cite{lan2020}; see also \cite{sunsddpc}, which proves the complexity of an SDDP-type variant with penalization for some nonconvex problems. An important open question is whether analogous complexity bounds for SDDP can be made dimension-free, namely independent of the dimensions of the state vectors, by extending the tools developed in this paper to the stochastic setting.



\appendix

\section{Technical lemmas}\label{App:tech}

\begin{lemma}\label{subgradValueFunc1}
Let
$G: \mathbb{R}^m \times \mathbb{R}^n \rightarrow \mathbb{R}$ be a convex
function and let $Y$ be a convex
compact set. Then, the value function 
 defined as
\begin{equation}\label{pbdefsc0}
            \V(x) := \min_{y\in Y} \, G(y,x) \quad \forall \ x \in \R^n
\end{equation}
is convex, finite everywhere, and $\partial \V(x) \ne \emptyset$ for every $x \in \R^n$.
Moreover,  the following three conditions are equivalent:
\begin{itemize}
    \item[a)] 
    $s \in \partial \V(x)$;
    \item[b)] there exists an optimal solution $y_x$ of \eqref{pbdefsc0} such that $(0,s) \in \partial G(y_x,x)$, in which case
\[
s \in \partial_x G(y_x,x);
\]
\item[c)]
$s$
is a Lagrange multiplier for
the constraint $z=x$ in the following reformulation of
the value function $\V$:
\begin{align}\label{pbdefsc1}
            \V(x) = \min \{ G(y,z): y \in Y,\; z=x\}.
            \end{align}
\end{itemize}
\end{lemma}
\begin{proof}
It follows from Proposition 3.3.3 of \cite{bertsekas2009convex} that $\V$ is a convex function. Since for every $x \in \R^{n}$, $G(\cdot,x)$ is continuous on the
compact set $Y$, function $\V$ is finite-valued
everywhere and therefore $\partial \V(x) \neq \emptyset$.

The equivalence between a) and c) follows from \cite[Theorem 9.3]{rockafellar1993lagrange}, as \eqref{pbdefsc0} has a feasible solution by compactness of $Y$ and $\V$ is lower semicontinuous since it is finite and convex.
Using Theorem 24(a) of \cite{rockafellar1974conjugate}, we also have
that a) is equivalent to the existence of
an optimal solution $y_x$  of \eqref{pbdefsc0} such that $(0,s)\in \partial G(y_x,x)$; equivalently, for all $(y,z)$,
$$
G(y,z)\ge G(y_x,x)+\left \langle 0,\,y-y_x\right \rangle+\left \langle s,\,z-x\right \rangle
= G(y_x,x)+\left \langle s,\,z-x\right \rangle.
$$
In particular, setting $y=y_x$ gives 
$$
G(y_x,z)\ge G(y_x,x)+\left \langle s,\,z-x\right \rangle
$$
for all $z$,
which is precisely the subgradient inequality characterizing $s\in \partial_x G(y_x,x)$.
\end{proof}

\begin{lemma}\label{subgradValueFunc2}
Let $X \subset \mathbb{R}^n$ 
and $Y \subset \mathbb{R}^m$
be convex compact sets and let
$F: \mathbb{R}^m \times \mathbb{R}^n \rightarrow \mathbb{R}$ and 
$\Gamma: \mathbb{R}^m \rightarrow \mathbb{R}$ be convex
functions.
 Then, the value function 
 defined as
\begin{equation}\label{pbdefsc2}
            \V(x) := \min_{y\in Y} \, \left\{G(y,x):=F(y,x) + \Gamma(y) \right\} \quad \forall \ x \in \R^n,
\end{equation}
is convex, finite everywhere, and $\partial \V(x) \ne \emptyset$ for every $x \in \R^n$.

If, in addition,
there exists $0 \leq M < +\infty$
such that
\begin{equation}\label{syx}
    \|s_x(y,x)\| \le M,
    \end{equation}
for every $(y,x) \in Y \times X$ 
and every subgradient 
$s_x(y,x) \in \partial_x F(y,x)$, 
    then $\V(\cdot)$
    restricted to $X$
        is $M$-Lipschitz continuous
        or equivalently $\|s\| \le M$ for every  $s \in \partial \V(x)$ and $x \in X$.
\end{lemma}

\begin{proof} We have that $\V$ is convex, finite everywhere, and $\partial \V(x) \ne \emptyset$ for every $x \in \R^n$ by a direct application of Lemma \ref{subgradValueFunc1} to $G(y,x):=F(y,x) + \Gamma(y).$ 

For the second part, assume that $\|s_x(y,x)\| \le M$, for every $(y,x) \in Y \times X$ and every subgradient $s_x(y,x) \in \partial_x F(y,x) = \partial_x G(y,x)$. Thus, for any $x \in X$ and $s \in \partial \V(x),$ Lemma \ref{subgradValueFunc1} shows that $s \in \partial_x G(y_x,x)$ for some $y_x \in Y$, so that $\|s\| \le M.$ 

To show the last equivalence, we proceed as follows. Assume $\|s\|\le M$ for all $x\in X$ and $s\in\partial \V(x)$. Then for any $x,x'\in X$ and $s\in\partial \V(x)$,
\[\V\left(x'\right)-\V(x)\ge \left\langle s,x'-x\right \rangle \ge -\|s\|\,\left\|x'-x\right\|_{*}\ge -M\left\|x'-x\right\|_{*}.\] 
Switching $x,x'$ yields
$|\V(x')-\V(x)|\le M\|x'-x\|_{*}$, i.e., $\V$ is $M$-Lipschitz on $X$ with respect to \ $\|\cdot\|_{*}$.
Conversely, if $\V$ is $M$-Lipschitz on $X$ w.r.t.\ $\|\cdot\|_{*}$, then for any $x\in X$, $s\in\partial \V(x)$, and $x'\in X$,
\[\left\langle s,x'-x\right \rangle \le \V\left(x'\right)-\V(x)\le M\left\|x'-x\right\|_{*},\] 
hence $\|s\|\le M$.
\end{proof}

\begin{lemma}\label{LemApp3}
Consider
\begin{equation}\label{probf}
f^\star:=\min\{f(x):Ax=b,\ x\in X\},
\end{equation}
and assume that its optimal solution set $X^\star$ is nonempty. Let $p^\star$ be a Lagrange multiplier associated with this problem. For $\rho>0$, define
\begin{equation}\label{regf}
    f_\rho^\star
:=
\min_{x\in X}
\left\{
f_\rho(x)
:=
f(x)+\frac{\rho}{2}\|Ax-b\|^2
\right\}.
\end{equation}
Let $\bar x\in X$ be a $\bar\varepsilon$-solution for problem \eqref{regf}. If
\begin{equation}\label{ineq:rho}
    \rho\ge \frac{4\|p^\star\|+8}{\bar\varepsilon},
\end{equation}
then
\begin{equation}\label{Lem3result}
f(\bar x)-f^\star\le \bar\varepsilon,
\qquad
\|A\bar x-b\|\le \bar\varepsilon.
\end{equation}
That is, $\bar x$ is a $\bar\varepsilon$-solution of problem \eqref{probf} in the sense that it has objective error at most $\bar\varepsilon$ and constraint violation at most $\bar\varepsilon$.
\end{lemma}

\begin{proof}
Let $x^\star\in X^\star$. Since $Ax^\star=b$, we have
\[
f_\rho^\star\le f_\rho(x^\star)=f(x^\star)=f^\star.
\]
Thus
\begin{equation}\label{ineq:frho}
    f(\bar x)-f^\star
+
\frac{\rho}{2}\|A\bar x-b\|^2
=
f_\rho(\bar x)-f^\star \le f_\rho(\bar x)-f_\rho^\star
\le \bar\varepsilon.
\end{equation}
Dropping the nonnegative penalty term gives the first inequality in \eqref{Lem3result}.
It follows from Corollary~2 in \cite{LanMonteiro2013} that for $x\in X$
\[
f(x)-f^\star \ge -\|p^\star\|\,\|Ax-b\|.
\]
Applying this inequality at $x=\bar x$ and setting
$r:=\|A\bar x-b\|$, we obtain
\[
-\|p^\star\|r+\frac{\rho}{2}r^2 \le f(\bar x)-f^\star + \frac{\rho}{2}\|A\bar x-b\|^2 \stackrel{\eqref{ineq:frho}}\le \bar\varepsilon.
\]
Hence
\[
r
\le
\frac{\|p^\star\|+\sqrt{\|p^\star\|^2+2\rho\bar\varepsilon}}{\rho}
\le
\frac{2\|p^\star\|}{\rho}
+
\sqrt{\frac{2\bar\varepsilon}{\rho}}.
\]
It follows from \eqref{ineq:rho} that
\begin{equation*}\label{Lem3result2}
r\le
\bar\varepsilon
\left(
\frac{2\|p^\star\|}{4\|p^\star\|+8}
+
\sqrt{\frac{2}{4\|p^\star\|+8}}
\right)
\le \bar\varepsilon,
\end{equation*}
which gives the second inequality in \eqref{Lem3result}.
\end{proof}

\section{Proof of Proposition~\ref{twocutsp}}


First, noting that $\overline \Gamma_t^{k-1}$ is an affine function satisfying $\overline \Gamma_t^{k-1} \le \Q_t$ and $\V_t(\cdot;\Gamma_{t+1}^k) \in {\cal B}_{X_t}(\Q_t)$, it thus follows from \eqref{def:Atk} that
\[
\overline \Gamma_t^k(\cdot) \le \beta_t^k \Q_t(\cdot) + \left(1-\beta_t^k\right) \V_t\left(\cdot;\Gamma_{t+1}^k\right) \le \beta_t^k \Q_t(\cdot) + \left(1-\beta_t^k\right) \Q_t(\cdot) = \Q_t(\cdot),
\]
and hence that $\overline \Gamma_t^k \in {\cal B}_{X_t}(\Q_t) $.

Next, we prove \eqref{eq:mirror} holds. The fact that $x_t^k$ and $\beta_t^k$ are primal and dual solutions of \eqref{eq:vtk} implies that
\begin{align}
    &m_{t-1}^k = F_t\left(x_t^k,x_{t-1}^k\right) + \beta_t^k \overline \Gamma_t^{k-1}\left(x_t^k\right) + \left(1-\beta_t^k\right) \ell_t^k\left(x_t^k\right), \label{eq:value} \\
    &0 \in \partial F_t\left(\cdot,x_{t-1}^k\right)\left(x_t^k\right) + \beta_t^k \nabla \overline \Gamma_t^{k-1}\left(x_t^k\right) + \left(1-\beta_t^k\right) s_t^{k-1}. \label{eq:optcond}
\end{align}
Using \eqref{def:vtk}, we also have
$m_{t-1}^k = F_t(x_t^k,x_{t-1}^k) + \Gamma_t(x_t^k)$,
which together with \eqref{def:Atk} and \eqref{eq:value} yields
\[
\overline \Gamma_t^k\left(x_t^k\right) \stackrel{\eqref{def:Atk}}= \beta_t^k \overline \Gamma_t^{k-1}\left(x_t^k\right) + \left(1-\beta_t^k\right) \ell_t^k\left(x_t^k\right) \stackrel{\eqref{eq:value}}= \Gamma_t^k\left(x_t^k\right),
\]
so the first identity of \eqref{eq:mirror} holds.
Moreover, \eqref{def:Atk} and \eqref{eq:optcond} show that 
\[
x_t^k = \underset{u_t\in X_t}\argmin F_t\left(u_t,x_{t-1}^k\right) + \overline \Gamma_t^k( u_t),
\]
which is the second identity of \eqref{eq:mirror}.
In view of Definition \ref{defmirror}, 
we have thus proved the proposition.

\bibliographystyle{plain}
\bibliography{ref}

\end{document}

%% file: def.tex
\usepackage{amsmath}
\usepackage{amsfonts}
\usepackage{latexsym}
\usepackage{amssymb}

\newtheorem{thm}{Theorem}[section]
\newtheorem{theorem}[thm]{Theorem}

\newtheorem{lemma}[thm]{Lemma}

\newtheorem{proposition}[thm]{Proposition}

\newtheorem{definition}[thm]{Definition}

\newtheorem{remark}[thm]{Remark}

\newcommand{\beq}{\begin{equation}}
\newcommand{\eeq}{\end{equation}}
\newcommand{\beqa}{\begin{eqnarray}}
\newcommand{\eeqa}{\end{eqnarray}}
\newcommand{\beqas}{\begin{eqnarray*}}
\newcommand{\eeqas}{\end{eqnarray*}}
\newcommand{\bi}{\begin{itemize}}
\newcommand{\ei}{\end{itemize}}

\newcommand{\vgap}{\vspace{.1in}}

\newcommand{\nn}{\nonumber}

\newcommand{\R}{\mathbb{R}}

\newcommand{\inner}[2]{\langle #1,#2\rangle}

\newcommand{\dom}{\mathrm{dom}\,}

\newcommand{\argmin}{\mathrm{argmin}\,}

\newcommand{\bConv}[1]{\overline{\mbox{\rm Conv}}\,(\R^{#1})}

\newcommand{\mConv}[1]{\overline{\mbox{\rm Conv}}_\mu\,(\R^{#1})}